\documentclass[a4paper,twoside,10pt]{article}
\usepackage[a4paper,left=2cm,right=2cm, top=3cm, bottom=3cm]{geometry}
\usepackage{cuted}
\usepackage{epsfig}
\usepackage{amsfonts,amsmath}
\usepackage{graphicx}
\usepackage{amssymb}
\usepackage{enumitem}
\usepackage{url}
\usepackage{booktabs}
\usepackage{cite}
\usepackage{amsfonts}
\usepackage{nicefrac}
\usepackage{microtype}
\usepackage{graphicx}
\usepackage{xcolor}
\usepackage{algpseudocode}
\usepackage{algorithm}
\usepackage{enumitem}
\usepackage{pgfplots}
\pgfplotsset{compat=1.18}
\usepackage{amsmath}
\usepackage{amsthm}
\usepackage{dsfont}
\usepackage{mathtools}
\usepackage{tcolorbox}
\usepackage{amsmath}
\usepackage{physics}
\usepackage{bbm}
\usepackage{wasysym}
\usepackage{caption}
\usepackage{subcaption}
\usepackage[hidelinks]{hyperref}
\hypersetup{
    colorlinks = true,
    linktoc=section,
    citecolor=red,
    linkcolor=blue,
    pdfborder={0 0 0}
}
\usepackage{orcidlink}

\newtheorem{theorem}{Theorem}[section]
\newtheorem{assumption}[theorem]{Assumption}
\newtheorem{lemma}[theorem]{Lemma}
\newtheorem{proposition}[theorem]{Proposition}

\newtheorem{remark}[theorem]{Remark}

\newtheorem{example}[theorem]{Example}

\definecolor{darkgreen}{rgb}{0,0.5,0}
\usepackage{mathtools}
\usepackage{comment}

\newcommand\coolover[2]{\mathrlap{\smash{\overbrace{\phantom{%
    \begin{matrix} #2 \end{matrix}}}^{\mbox{$#1$}}}}#2}

\newcommand\coolunder[2]{\mathrlap{\smash{\underbrace{\phantom{%
    \begin{matrix} #2 \end{matrix}}}_{\mbox{$#1$}}}}#2}

\newcommand\coolleftbrace[2]{%
#1\left\{\vphantom{\begin{matrix} #2 \end{matrix}}\right.}

\newcommand\coolrightbrace[2]{%
\left.\vphantom{\begin{matrix} #1 \end{matrix}}\right\}#2}

\usepackage{amsmath}
\usepackage{tikz}
\usetikzlibrary{calc}

\usepackage{pgfplots}
\usepgfplotslibrary{groupplots}
\usepackage{amsmath}
\usepackage{tikz}
\usepackage{amsfonts}
\usepackage{nicefrac}
\usepackage{soul}
\usepackage{oplotsymbl}
\pgfplotsset{
  log x ticks with fixed point/.style={
      xticklabel={
        \pgfkeys{/pgf/fpu=true}
        \pgfmathparse{exp(\tick)}%
        \pgfmathprintnumber[fixed relative, precision=3]{\pgfmathresult}
        \pgfkeys{/pgf/fpu=false}
      }
  },
  log y ticks with fixed point/.style={
      yticklabel={
        \pgfkeys{/pgf/fpu=true}
        \pgfmathparse{exp(\tick)}%
        \pgfmathprintnumber[fixed relative, precision=3]{\pgfmathresult}
        \pgfkeys{/pgf/fpu=false}
    }
  }
}

\newcommand{\logLogSlopeTriangle}[5]
{
    \pgfplotsextra
    {
        \pgfkeysgetvalue{/pgfplots/xmin}{\xmin}
        \pgfkeysgetvalue{/pgfplots/xmax}{\xmax}
        \pgfkeysgetvalue{/pgfplots/ymin}{\ymin}
        \pgfkeysgetvalue{/pgfplots/ymax}{\ymax}

        \pgfmathsetmacro{\xArel}{#1}
        \pgfmathsetmacro{\yArel}{#3}
        \pgfmathsetmacro{\xBrel}{#1-#2}
        \pgfmathsetmacro{\yBrel}{\yArel}
        \pgfmathsetmacro{\xCrel}{\xArel}

        \pgfmathsetmacro{\lnxB}{\xmin*(1-(#1-#2))+\xmax*(#1-#2)} 
        \pgfmathsetmacro{\lnxA}{\xmin*(1-#1)+\xmax*#1} 
        \pgfmathsetmacro{\lnyA}{\ymin*(1-#3)+\ymax*#3} 
        \pgfmathsetmacro{\lnyC}{\lnyA+#4*(\lnxA-\lnxB)}
        \pgfmathsetmacro{\yCrel}{(\lnyC-\ymin)/(\ymax-\ymin)} 

        \coordinate (A) at (rel axis cs:\xArel,\yArel);
        \coordinate (B) at (rel axis cs:\xBrel,\yBrel);
        \coordinate (C) at (rel axis cs:\xCrel,\yCrel);

        \draw[#5]   (A)-- node[pos=0.5,anchor=north] {}
                    (B)-- 
                    (C)-- node[pos=0.5,anchor=west,] {\scriptsize$\boldsymbol{#4}$}
                    cycle;
    }
}

\usepackage{tikz}

\usepackage{tikz}
\newcommand{\squareasterisk}{%
  \begin{tikzpicture}[baseline=-0.5ex]
    \draw (-0.1,-0.1) rectangle (0.6em, 0.6em);
    \node at (0.06,0.06) {$*$};
  \end{tikzpicture}%
}

\DeclareFontFamily{U}{matha}{\hyphenchar\font45}
\DeclareFontShape{U}{matha}{m}{n}{
<-6> matha5 <6-7> matha6 <7-8> matha7
<8-9> matha8 <9-10> matha9
<10-12> matha10 <12-> matha12
}{}
\DeclareSymbolFont{matha}{U}{matha}{m}{n}
\DeclareFontFamily{U}{mathx}{\hyphenchar\font45}
\DeclareFontShape{U}{mathx}{m}{n}{
<-6> mathx5 <6-7> mathx6 <7-8> mathx7
<8-9> mathx8 <9-10> mathx9
<10-12> mathx10 <12-> mathx12
}{}
\DeclareSymbolFont{mathx}{U}{mathx}{m}{n}
\DeclareMathDelimiter{\vvvert} {0}{matha}{"7E}{mathx}{"17}%

\DeclarePairedDelimiterX{\normiii}[1]
{\vvvert}
{\vvvert}
{\ifblank{#1}{\:\cdot\:}{#1}}

\renewcommand{\phi}{\varphi}
\newcommand{\mi}{\mathrm{i}}
\DeclareMathOperator{\thh}{th}

\newcommand{\B}{\mathbf{B}}
\newcommand{\bx}{\boldsymbol{x}}
\newcommand{\bh}{\boldsymbol{h}}
\newcommand{\C}{\mathbb{C}}
\newcommand{\A}{\mathbf{A}}
\newcommand{\bS}{\mathbf{S}}
\newcommand{\M}{\mathbf{M}}
\newcommand{\K}{\mathbf{K}}

\newcommand{\R}{\mathbb{R}}
\renewcommand{\P}{\mathbb{P}}
\newcommand{\N}{\mathbb{N}}
\newcommand{\bM}{\mathbf{M}}

\newcommand{\bC}{\mathbf{C}}
\newcommand{\Z}{\mathbb{Z}}
\newcommand{\calD}{\mathcal{D}}

\newcommand{\QT}{Q_T}
\newcommand{\nablax}{\nabla_{\bx}}
\newcommand{\hx}{h_{\bx}}
\newcommand{\dpt}{\partial_t}
\newcommand{\tj}{t_j}
\newcommand{\tjmo}{t_{j-1}}

\newcommand{\Sht}[2]{S_{h_t}^{#1}}
\newcommand{\Shx}{S_{h_{\boldsymbol{x}}}}

\newcommand{\wht}{w_{h_t}^p}
\newcommand{\Uhp}{\Psi_{\boldsymbol{h}}^p}

\newcommand{\Whpmo}{W_{\boldsymbol{h}}^{p-1}}
\newcommand{\calE}{\mathcal{E}}
\newcommand{\calM}{\mathcal{M}}
\newcommand{\conj}[1]{\overline{#1}}
\newcommand{\Norm}[2]{\|#1\|_{#2}}

\newcommand{\eremk}{\hbox{}\hfill\rule{0.8ex}{0.8ex}}

\numberwithin{equation}{section}

\title{A matrix-based approach to the stability of a space--time \\ isogeometric method for the linear Schr\"odinger equation}
\author{Matteo~Ferrari\,\orcidlink{0000-0002-2577-1421}\thanks{Faculty of Mathematics, Universit\"at Wien, Oskar-Morgenstern-Platz 1, 1090 Vienna, Austria (\href{mailto:matteo.ferrari@univie.ac.at}{matteo.ferrari@univie.ac.at}).}  \and Sergio G\'omez\,\orcidlink{0000-0001-9156-5135}\thanks{Department of Mathematics and Applications, University of Milano-Bicocca, Via Cozzi 55, 20125 Milan, Italy (\href{mailto:sergio.gomezmacias@unimib.it}{sergio.gomezmacias@unimib.it}).}
\thanks{IMATI-CNR ``E. Magenes", Via Ferrata 5, 27100 Pavia, Italy} }

\date{}

\begin{document}
\maketitle

\begin{abstract}
\noindent 
We propose a space--time isogeometric finite element method for the linear Schr\"odinger equation, and establish its unconditional stability through a matrix-based analysis. Although maximal-regularity splines in time provide higher accuracy per degree of freedom compared to piecewise continuous polynomials, the nonlocal support of the spline bases precludes the use of standard variational arguments in the stability proofs. To overcome this, we show that the resulting scheme is governed by a family of nearly Toeplitz system matrices and, by studying the condition number of these matrices, we prove that the family is weakly well-conditioned, which guarantees the unconditional stability of the method. Furthermore, the discrete scheme preserves mass and energy at the final time. Numerical experiments confirm our theoretical findings and illustrate the optimal convergence behavior of the scheme. Finally, we exploit an algebraic connection between our formulation and a recent first-order-in-time space--time isogeometric method for the wave equation to derive a complete matrix-based stability analysis for the latter.
\end{abstract}


\section{Introduction}

In this paper, we focus on the dimensionless linear time-dependent Schr\"odinger equation in $d$ space dimensions with a spatially varying potential. Let~$\QT$ be a space--time cylinder given by~$\Omega \times J_T$, where~$\Omega \subset \R^d$ ($d \in \N)$ is an open, bounded domain with Lipschitz boundary~$\partial \Omega$, and~$J_T := (0, T)$, for some final time~$T > 0$. Given a real-valued potential~$V : \Omega \rightarrow \R$, an initial datum $\Psi_0 : \Omega \to \mathbb{C}$, and an external forcing term~$F : \QT \rightarrow \C$, we consider the following linear Schr\"odinger initial and boundary value problem (IBVP): find~$\Psi : \overline{Q}_T \to \C$ such that
\begin{equation}
\label{eq:1.1}
    \begin{cases}
        \mi \dpt \Psi(\bx,t) + \frac{1}{2} \Delta_{\bx} \Psi (\bx,t) + V(\bx) \Psi (\bx,t) = F(\bx,t) & (\bx,t) \in \QT,
        \\ \Psi(\bx,t) = 0 & (\bx,t) \in \partial \Omega \times J_T,
        \\ \Psi(\bx,0) = \Psi_0(\bx) & \bx \in \Omega.
    \end{cases}
\end{equation}
Above, $\mi$ is the imaginary unit, and~$\Delta_{\bx}$ denotes the spatial Laplacian operator.

Model~\eqref{eq:1.1} may describe the evolution of a quantum wave function~$\Psi$ under the influence of a spatially varying potential~$V$.  In particular, this equation governs the behavior of a nonrelativistic quantum particle in~$d$ space dimensions, where the term~$\frac{1}{2} \Delta_{\bx} \Psi$ represents the kinetic energy contribution, and~$V\Psi$ accounts for the interaction with the potential field.

\paragraph{Previous works.}
Space--time finite element methods (FEM) treat time as an additional space dimension in time-dependent PDEs. These methods offer several advantages, as they generally allow for high-order approximations in both space and time, can be formulated in a variational setting that closely resembles the continuous one, and can be naturally combined with efficient solvers. 

The design of space--time methods for the Schr\"odinger equation has received significant attention in recent years.  In~\cite{KarakashianMakridakis1998}, the cubic nonlinear Schr\"odinger equation was discretized by combining conforming FEM space discretizations with the discontinuous Galerkin (dG) time-stepping scheme. The continuous Petrov--Galerkin time discretization in~\cite{AzizMonk1989} for the heat equation was also later studied in~\cite{KarakashianMakridakis1999} for the cubic nonlinear Schr\"odinger equation. The more refined analysis presented in~\cite{DodingHenning2024} of the method in~\cite{KarakashianMakridakis1999} is valid for two- and three-dimensional space domains, without requiring artificial restrictions on the time steps. A recent approach based on a conforming space--time Galerkin discretization is shown to be effective in \cite{Zank2025}. A space--time discontinuous Petrov--Galerkin method for the linear Schr\"odinger equation was proposed in~\cite{DemkowiczGopalakrishnanNagarajSepulveda2017}, which is based on an ultraweak formulation with test space chosen to ensure \emph{inf-sup stability}. Further space--time methods based on ultraweak formulations include spline-based~\cite{HainUrban2024}, least-squares \cite{BressanKushovaSangalliTani2023}, and discontinuous Galerkin \cite{Gomez_Moiola:2022,Gomez_Moiola:2024,Gomez_Moiola_Perugia_Stocker:2023} methods.

This work focuses on conforming space--time finite element methods where the time variable is discretized using splines with maximal regularity. The use of such spaces allows us to obtain approximations with accuracy comparable to methods based on standard polynomial spaces but with fewer degrees of freedom \cite{EvansBazilevsBabuskaHughes2009, BressanSande2019, SandeManniSpeleers2019}. However, employing maximal-regularity splines introduces significant challenges in the stability analysis, as standard variational arguments (see, e.g., the recent review~\cite{Gomez:2026b}) cannot be applied. This difficulty primarily arises from the extended support resulting from their high regularity, which differ from the localized nature of standard piecewise continuous polynomials. 
To overcome this limitation, we carry out a matrix-based stability analysis that exploits the nearly Toeplitz structure of the resulting temporal part of the discrete system. This algebraic structure is a direct consequence of employing B-splines over a uniform temporal grid. 

\paragraph{Novelties.}
In this paper, we propose and analyze a space--time FEM to approximate the solution to the linear Schr\"odinger IBVP~\eqref{eq:1.1}. The proposed method combines a Petrov--Galerkin scheme in time based on splines with maximal regularity with $H_0^1(\Omega)$-conforming space discretizations. 
Moreover, we show that the proposed method conserves mass and energy at the final time for any spatial discretization. In the spirit of~\cite{KarakashianMakridakis1999}, the discrete test space is constructed by taking the temporal derivative of the trial space. However, while the stability analysis in~\cite{KarakashianMakridakis1999,DodingHenning2024} relied heavily on the specific properties of discontinuous polynomials, such variational techniques are not applicable in our setting due to the nonlocal support of B-splines. Therefore, along the lines of~\cite{FerrariFraschiniLoliPerugia2025, FerrariFraschini2026}, we address the stability analysis by studying the condition number of the family of matrices stemming from the time discretization. 
This analysis combines two main tools: properties of the symbol of spline-based discretization matrices~\cite{GaroniManniPelosiSerraCapizzanoSpeleers2014,Donatelli2016,EkstromFurciGaroniManniSerraCapizzanoSpeleers2018,Garoni2019}, and the behavior of the condition number of general Toeplitz
band matrices characterized in~\cite{AmodioBrugnano1996}. Our main theoretical advancements are the following. 

\begin{itemize}

\item We show unconditional stability of the proposed method (see Theorem~\ref{th:3.21}). To achieve this, we exploit the generalized eigenvalue problem associated with the discrete spatial Hamiltonian. By expanding the solution in terms of the corresponding eigenfunctions, we decouple the fully discrete formulation into a set of independent temporal ODEs, which are then analyzed through our matrix-based approach.

\item We establish an algebraic connection between this method and the first-order-in-time variational formulation for the wave equation presented and analyzed in~\cite{FerrariFraschiniLoliPerugia2025} (see Section~\ref{sec:4}). This connection allows us to provide a proof of unconditional stability for the latter method, thereby eliminating the need for the numerical verifications required in the previous analysis reported in \cite{FerrariFraschiniLoliPerugia2025}.
\end{itemize}

\paragraph{Structure of the remainder of the paper.} 
The rest of the paper is organized as follows. In the following section, we introduce the continuous weak formulation of the linear Schrödinger IBVP and define the proposed space--time Galerkin method based on splines with maximal temporal regularity. We also prove that the discrete scheme preserves mass and energy at the final time. In Section~\ref{sec:3}, we carry out a stability analysis of the fully discrete method. To this end, we study a related variational formulation for a time-dependent scalar problem and analyze the structure and conditioning of the resulting discrete system. We show that the associated system matrices are nearly Toeplitz and, by analyzing their corresponding polynomials symbol, we prove the unconditional stability of the method. In Section~\ref{sec:4}, we highlight an algebraic connection between our variational formulation for the Schr\"odinger equation and that of a first-order-in-time isogeometric method for the wave equation. We thus provide a new, purely theoretical proof of unconditional stability for the method introduced in \cite{FerrariFraschiniLoliPerugia2025}, thereby removing the need for numerical verification. Finally, in Section~\ref{sec:5}, we present numerical experiments that validate our theoretical results. We describe an efficient implementation of the solver based on the \textit{Bartels--Stewart method}, demonstrate optimal convergence rates in space and time, and illustrate the conservation of physical invariants.

\paragraph{Notation for function spaces.} 
Here and throughout the paper, we use standard notation for~$L^p$, Sobolev, and Bochner spaces, which can be found, e.g., in~\cite{Brezis2010}. In particular, given a bounded domain~$\calD \subset \R^n$ ($n \in \N$), $L^2(\calD)$ denotes the space of complex-valued functions that are Lebesgue square integrable over~$\calD$. The associated sesquilinear inner product~$(u,v)_\calD := \int_\calD u(\bx) \overline{v}(\bx) \dd \bx$, with $\overline{v}$ denoting the complex conjugate of~$v$, induces the $L^2(\calD)$ norm $\| u \|^2_{\calD} := (u,u)_\calD$. Given~$s > 0$, we denote by~$H^s(\calD)$ the Sobolev space of order~$s$, by~$H_0^1(\calD)$ the closure of~$C_0^{\infty}(\calD)$ in the~$H^1(\calD)$ norm, and by~$H^{-1}(\calD)$ the dual space of~$H_0^1(\calD)$.

Given a time interval~$(a, b)$ and a Banach space~$(X, \Norm{\cdot}{X})$, we denote by~$L^2(a, b; X)$ the corresponding Bochner space with norm
\begin{equation*}
    \Norm{v}{L^2(a, b; X)}^2 := \int_a^b \Norm{v(\cdot, t)}{X}^2 \dd t.
\end{equation*}
Moreover, we define
\begin{equation*}
    H^1(a, b; X) := \big\{v \in L^2(a, b; X) \ : \  \dpt v \in L^2(a, b; X) \big\}.
\end{equation*}

\section{Continuous weak formulation and space--time method}
Let in \eqref{eq:1.1} the forcing term~$F \in L^2(\QT) \cap H^1(0, T; H^{-1}(\Omega))$, the real-valued potential~$V \in L^{\infty}(\Omega)$, and the initial condition~$\Psi_0 \in H_0^1(\Omega)$. The continuous weak formulation of~\eqref{eq:1.1} reads: find~$\Psi \in C^0([0,T]; H_0^1(\Omega))$ with~$\dpt \Psi \in C^{0}([0, T]; H^{-1}(\Omega))$ such that~$\Psi(\cdot, 0) = \Psi_0$ in~$H_0^1(\Omega)$ and
\begin{equation} \label{eq:2.1}
\begin{split}
        \int_0^T \Big(\mi \langle \dpt \Psi(\cdot,t),W(\cdot,t) \rangle & - \frac{1}{2} (\nabla_{\bx} \Psi(\cdot,t), \nabla_{\bx} W(\cdot,t))_{\Omega} + (V \Psi(\cdot,t), W(\cdot,t))_{\Omega} \Big) \dd t 
        \\ & = \int_{0}^T (F(\cdot,t), W(\cdot,t))_{\Omega} \, \dd t  \qquad  \quad \text{for all~} W \in L^2(0, T; H^1_{0}(\Omega)),
\end{split}
\end{equation}
where~$\nablax$ is the spatial gradient operator, and~$\langle \cdot, \cdot \rangle$ denotes the duality between~$H^{-1}(\Omega)$ and~$H_0^1(\Omega)$. According to~\cite[Thm.~10.1 and Rem.~10.2 in Ch.~3]{Lions-MagenesVol1:1972}, there exists a unique solution to the continuous weak formulation~\eqref{eq:2.1}.

We define the energy functional
\begin{equation} \label{eq:2.2}
    \mathcal{E}_{\Psi}(t) := -\frac{1}{2} \int_\Omega |\nabla_{\bx} \Psi(\bx,t)|^2 \, \dd \bx + \int_\Omega V(\bx) | \Psi(\bx,t)|^2  \, \dd \bx,
\end{equation}
and the mass functional
\begin{equation} \label{eq:2.3}
    \mathcal{M}_{\Psi}(t) := \int_\Omega |{\Psi}(\bx,t)|^2  \, \dd \bx.
\end{equation}
\begin{remark}[Conservation of invariants] \label{rem:21}
Assuming that~$F \in H^1(0, T; L^2(\Omega))$ and~$\partial \Omega$ is sufficiently smooth, the solution~${\Psi}$ to~\eqref{eq:2.1} belongs to the space~$L^2(0, T; H_0^1(\Omega) \cap H^2(\Omega)) \cap H^1(0, T; L^2(\Omega))$; see~\cite[Thm.~12.1 in Ch.~5]{Lions-MagenesVol2:1972}. Then, we can choose~$W$ in the weak formulation~\eqref{eq:2.1} as~$W(\bx,s) = \partial_s {\Psi}(\bx,s) \mathds{1}_{[0,t]}(s)$, where~$\mathds{1}_{[0,t]}$ is the characteristic function on~$[0,t]$, and take the real part of the resulting equation to obtain
\begin{align*}
    \calE_{\Psi}(t) & = \calE_{\Psi}(0) + 2 \Re \int_0^t F(\bx,s) \partial_s \conj{{\Psi}}(\bx,s)\, \dd \bx \, \dd s \quad \text{~for all~} t \in [0,T].
\end{align*}
Similarly, choosing~$W$ in~\eqref{eq:2.1} as $W(\bx,s) = {\Psi}(\bx,s) \mathds{1}_{[0,t]}(s)$, taking the imaginary part of the resulting equation, and using the identity
\begin{equation} \label{eq:2.4}
    \Re\, (\dpt {\Psi}(\cdot,t), {\Psi}(\cdot,t))_\Omega = \frac{1}{2} \int_\Omega | {\Psi}(\bx,t) |^2 \dd \bx - \frac{1}{2} \int_\Omega | {\Psi}_0(\bx) |^2 \dd \bx
\end{equation}
leads to
\begin{align*}
    \calM_{\Psi}(t) & = \calM_{\Psi}(0) + 2 \Im \int_0^t F(\bx,s)  \conj{{\Psi}}(\bx,s)\,  \dd \bx \, \dd s\quad \text{~for all~} t \in [0,T].
\end{align*}
As a consequence, for~$F \equiv 0$, the energy~\eqref{eq:2.2} and the mass~\eqref{eq:2.3} functionals are conserved, i.e., 
\begin{equation*}
    \frac{\dd}{\dd t} \calE_{\Psi}(t) = 0 \quad \text{ and } \quad \frac{\dd}{\dd t} \calM_{\Psi}(t) = 0.
\end{equation*}
It is desired to reproduce this property at the discrete level.
\eremk
\end{remark}

\subsection{Definition of the space--time method} \label{sec:21}
Given~$N_t \in \N$, we define a uniform mesh $\{ t_j = jh_t \, : \, j = 0, \ldots, N_t \}$ for the time interval~$(0, T)$, with meshsize~$h_t = T/N_t$.  Moreover, for~$j = 1, \ldots, N_t$, we define~$I_j := (\tjmo, \tj)$. We consider the space of maximal regularity splines of degree $p \geq 1$ and regularity $p - 1$ with values in the field $\mathbb{K}$, defined by
\begin{equation*}
    S_{h_t}^p(0,T;\mathbb{K}) := \{ \wht \in C^{p-1}([0,T]) \ : \ \wht {}_{|_{I_j}} \in \P_p(I_j; \mathbb{K}), \, j = 1,\ldots N_t\},
\end{equation*}
where~$\P_p(I_j; \mathbb{K})$ is the space of $\mathbb{K}$-valued polynomials of degree~$p$ defined on~$I_j$. This space satisfies 
\begin{equation*}
    \dim_{\mathbb{K}}(S_{h_t}^p(0,T;\mathbb{K})) = N_t + p.
\end{equation*}
We also define the following subspaces with null initial and final conditions, respectively,
\begin{equation} \label{eq:2.5}
\begin{aligned}
    S_{h_t,0,\bullet}^p(0,T;\mathbb{K}) & = \big\{ \wht \in S_{h_t}^p(0,T;\mathbb{K}) : \wht (0) = 0\big\}, 
    \\ S_{h_t,\bullet,0}^p(0,T;\mathbb{K}) & = \big\{ \wht \in S_{h_t}^p(0,T;\mathbb{K}): \wht(T) = 0\big\}.
\end{aligned}
\end{equation}
Let~$\Shx(\Omega)$  be a discrete conforming subspace of~$H_0^1(\Omega)$, and~$\Pi_{\hx} {\Psi}_0 \in \Shx(\Omega)$ be an approximation of~${\Psi}_0$.

\begin{remark}[Choice of the discrete spaces]
While the proposed method offers great flexibility in the choice of the spatial discretization,   as it allows for any conforming discrete space and can be easily generalized to accommodate nonconforming spatial discretizations, the choice of the discrete spaces for the temporal discretization is more delicate and has been chosen carefully. This distinction arises from the analytical difficulties  inherent to the treatment of the time variable in this context. Indeed, developing a continuous space--time variational formulation for the Schr\"odinger equation that yields unconditionally stable schemes for arbitrary choices of discrete temporal spaces remains an open problem. \eremk
\end{remark}
Denoting by~$\otimes$ the algebraic tensor product of vector spaces, the proposed method, which is a discrete counterpart of the continuous weak formulation~\eqref{eq:2.1}, reads:
\begin{tcolorbox}[
    colframe=black!50!white,
    colback=blue!5!white,
    boxrule=0.5mm,
    sharp corners,
    boxsep=0.5mm,
    top=0.5mm,
    bottom=0.5mm,
    right=0.25mm,
    left=0.1mm
]
    \begingroup
    \setlength{\abovedisplayskip}{0pt}
    \setlength{\belowdisplayskip}{0pt}
    \begin{equation} \label{eq:2.6}
    \begin{aligned}
    &\hspace{-1.75cm}\text{find}~\Uhp \in  \Pi_{\hx} \Psi_0 + S_{h_t,0,\bullet}^p(0,T;\C) \otimes \Shx(\Omega)~\text{such that} \vspace{0.1cm}
    \\ & \mi (\dpt \Uhp, \Whpmo)_{\QT} - \frac{1}{2} (\nablax \Uhp, \nablax \Whpmo)_{\QT} + (V \Uhp, \Whpmo)_{\QT} = (F, \Whpmo)_{\QT}
    \vspace{0.1cm}
    \\ & \hspace{-1.8cm}\text{~for all~} \Whpmo \in \Sht{p-1}{p-2}(0,T;\C) \otimes \Shx(\Omega).
    \end{aligned}
    \end{equation}
    \endgroup
\end{tcolorbox}
As mentioned before, method~\eqref{eq:2.6} is related to the one introduced by Karakashian and Makridakis in~\cite{KarakashianMakridakis1999}, where continuous piecewise polynomials of degree~$p$ (resp. piecewise polynomials of degree~$p - 1$) in time are used for the trial (resp. test) space. The method in \cite{KarakashianMakridakis1999} conserves the energy~\eqref{eq:2.2} and the mass~\eqref{eq:2.3} functionals at all discrete times. For our method, the extended support of the B-splines used as test functions prevents us from establishing this conservation property at all intermediate discrete time steps. Nevertheless, in Proposition~\ref{prop:2.3}, we show that method~\eqref{eq:2.6} conserves the energy and the mass at the final time~$T$.
\begin{proposition}[Conservation properties of~\eqref{eq:2.6}]
\label{prop:2.3}
Let~$F \equiv 0$ and~$\Uhp$ be a solution to method~\eqref{eq:2.6}. Then, it holds
\begin{equation*}
    \calM_{\Uhp}(T) = \calM_{\Uhp}(0) \qquad \text{ and } \qquad \calE_{\Uhp}(T) = \calE_{\Uhp}(0).
\end{equation*}
\end{proposition}
\begin{proof}
For the conservation of the mass, we choose~$\Whpmo = \Pi_{h_t}^{p-1} \Uhp$ in~\eqref{eq:2.6}, where~$\Pi_{h_t}^{p-1}$ is the $L^2(0, T)$-orthogonal projection into~$S_{h_t}^{p-1}(0,T;\C)$. Using the orthogonality properties of~$\Pi_{h_t}^{p-1}$, we obtain
\begin{align*}
        0 & = \mi (\dpt \Uhp, \Pi_{h_t}^{p-1} \Uhp)_{\QT} - \frac{1}{2} (\nablax \Uhp, \nablax \Pi_{h_t}^{p-1} \Uhp)_{\QT} + (V \Uhp, \Pi_{h_t}^{p-1} \Uhp)_{\QT}
        \\ &  = \mi (\dpt \Uhp, \Uhp)_{\QT} - \frac{1}{2} \|\nablax \Pi_{h_t}^{p-1} \Uhp\|^2_{L^2(\QT)} + (V \Pi_{h_t}^{p-1} \Uhp, \Pi_{h_t}^{p-1} \Uhp)_{\QT}.
\end{align*}
Taking the imaginary part of the above equation and using identity~\eqref{eq:2.4} 
yields the conservation of mass. The conservation of the energy readily follows by taking~$\Whpmo = \dpt \Uhp$ in~\eqref{eq:2.6} and proceeding as for the continuous case (see Remark~\ref{rem:21}).
\end{proof}

\section{Stability analysis} \label{sec:3}
In this section, we study the stability properties of the space--time method~\eqref{eq:2.6}. To do so, we first consider the corresponding discretization of the ODEs derived from the Fourier expansion of the  continuous solution to~\eqref{eq:2.1}. Then, the stability analysis for the space--time method~\eqref{eq:2.6} follows in an analogous way; see Remark~\ref{rem:32} below for more details.

Henceforth, we assume that the spatial domain~$\Omega$ is convex, $F \in H^1(0,T;L^2(\Omega))$, and  $\Psi_0 \in H^2(\Omega) \cap H_0^1(\Omega)$ to guarantee that the solution~$\Psi$ to \eqref{eq:2.1} belongs to~$L^2(0,T;H_0^1(\Omega)) \cap H^1(0,T;L^2(\Omega))$.

We also define the following Sobolev spaces:
\begin{equation*}
\begin{split}
    H_{0, \bullet}^1(0, T) & := \{v \in H^1(0, T) \, : \, v(0) = 0\}, 
    \\ H_{\bullet, 0}^1 (0, T) & := \{v \in H^1(0, T) \, : \, v(T) = 0 \}. 
    \end{split}
\end{equation*}

\subsection{Variational formulation of the associated initial value problem}

Consider the eigenvalue problem
\begin{equation*}
    \begin{cases}   
        \frac{1}{2} \Delta \Phi (\bx) + V(\bx) \Phi (\bx) = \mu \Phi(\bx) & \bx \in \Omega,
        \\ \Phi(\bx) = 0 & \bx \in \partial \Omega.
    \end{cases}
\end{equation*}
Due to the assumption~$V \in L^\infty(\Omega)$, according to~\cite[Rem.~30 in Ch.~9.8]{Brezis2010}, there exists a Hilbert basis of $L^2(\Omega)$ composed of eigenfunctions~$\{\Phi_j(\bx)\}_{j \ge 1}$ associated with real eigenvalues~$\{\mu_j \}_{j \ge 1}$ satisfying~$\mu_j \to +\infty$ as~$j \to +\infty$. Therefore, any solution to~\eqref{eq:2.1} admits the representation
\begin{equation*}
    \Psi(\bx,t) = \sum_{j=1}^\infty \psi_j(t) \Phi_j(\bx),
\end{equation*}
for some coefficients~$\{\psi_j\}_{j \geq 1} \subset H^1(0,T)$ satisfying
\begin{equation*}
    \begin{cases}
        \mi (\dpt \psi_j, w)_{J_T} + \mu_j ( \psi_j, w)_{J_T} = (f_j, w)_{J_T} \quad \text{for all~} w \in L^2(0,T), \\
        \psi_j(0) = \Psi_0^j,
    \end{cases}
\end{equation*}
where $f_j(t) := (F(\cdot,t), \psi_j)_\Omega$ and $\Psi_0^j := (\Psi_0, \psi_j)_\Omega$. Consequently, without loss of generality, we consider the following initial value problem (IVP): given~$f \in L^2(0,T)$, $\mu \in \R$, and~$\psi_0 \in \C$, find $\psi \in H^1(0,T)$ such that
\begin{equation} \label{eq:3.1}
    \begin{cases}
        \mi \dpt \psi(t) + \mu \psi(t) = f(t) & t \in (0,T), \\
        \psi(0) = \psi_0.
    \end{cases}
\end{equation}
To study the stability of the space--time method~\eqref{eq:2.6}, we focus on the corresponding discrete variational formulation associated with the IVP~\eqref{eq:3.1}, and analyze its stability with respect to the real parameter~$\mu$.

As in Section \ref{sec:21}, for a given~$N_{t} \in \N$, we define a uniform mesh $\{ t_j = j{h_t}\ : \ j = 0, \ldots, N_{t} \}$ for the time interval~$(0, T)$, with meshsize~${h_t} = T/N_{t}$. The numerical discretization of the IVP~\eqref{eq:3.1}
reads:
\begin{tcolorbox}[
    colframe=black!50!white,
    colback=blue!5!white,
    boxrule=0.5mm,
    sharp corners,
    boxsep=0.5mm,
    top=0.5mm,
    bottom=0.5mm,
    right=0.25mm,
    left=0.1mm
]
    \begingroup
    \setlength{\abovedisplayskip}{0pt}
    \setlength{\belowdisplayskip}{0pt}
    \begin{equation} \label{eq:3.2}
    \begin{aligned}
    &\hspace{-1.75cm}\text{find}~\psi^p_{h_t} \in  \psi_0 + S_{{h_t},0,\bullet}^p(0,T;\C)~\text{such that} \vspace{0.1cm}
    \\ & \mi (\dpt \psi^p_{h_t}, w^{p-1}_{h_t})_{J_T} + \mu (\psi^p_{h_t}, w^{p-1}_{h_t})_{J_T} = (f, w^{p-1}_{h_t})_{J_T} \quad \text{~for all~} w^{p-1}_{h_t} \in S^p_{h_t}(0,T;\C).
    \end{aligned}
    \end{equation}
    \endgroup
\end{tcolorbox}
\begin{remark}[Equivalent formulation] \label{rem:31}
Since $\dpt : S_{{h_t},\bullet,0}^p(0,T;\C) \to S^{p-1}_{h_t}(0,T;\C)$ is an isomorphism, integration by parts and the definition of the spaces~$S^p_{{h_t}, \bullet, 0}(0, T; \C)$ and~$S^p_{{h_t}, 0, \bullet}(0, T; \C)$ can be used to show that the discrete formulation \eqref{eq:3.2} is equivalent to: find~$\psi^p_{h_t} \in \psi_0 + S_{{h_t},0,\bullet}^p(0,T;\C)$ such that
\begin{equation*} \label{eq:11}
        \mi (\dpt \psi^p_{h_t}, \dpt w^p_{h_t})_{J_T} - \mu (\dpt \psi^p_{h_t}, w^p_{h_t})_{J_T} = (f, \dpt w^p_{h_t})_{J_T} + \mu \psi_0 w^p_{h_t}(0) \quad\quad \text{for all}~w^p_{h_t} \in S_{{h_t},\bullet,0}^p(0,T;\C).
\end{equation*}
\eremk
\end{remark}
\begin{remark}[Stability of the space--time method~\eqref{eq:2.6}]
\label{rem:32}
The theoretical results in this section can be easily adapted to the fully discrete setting in~\eqref{eq:2.6}. The reasoning is the following.
\begin{itemize}
\item Given a basis for the discrete space~$\Shx(\Omega)$, the conforming space discretization of the Hamiltonian operator~$(\frac12 \Delta_{\bx} + V)$ leads to the generalized eigenvalue problem: find pairs~$(\mu_{h_{\bx}}, \Phi_{h_{\bx}}) \in \C \times \Shx(\Omega)$ such that
\begin{equation*}
    \mathbf{A}_{h_{\bx}} \boldsymbol{\Phi}_{h_{\bx}} =  \mu_{h_{\bx}} \M_{h_{\bx}} \boldsymbol{\Phi}_{h_{\bx}},
\end{equation*}
where~$\mathbf{A}_{h_{\bx}}$ and~$\bM_{h_{\bx}}$ are the matrices representing, respectively, the Hamiltonian operator and the~$L^2(\Omega)$ inner product, and~$\boldsymbol{\Phi}_{h_{\bx}}$ is the vector representation of~$\Phi_{h_{\bx}}$ in the given basis.
\item Let $N_{\bx} := \dim_{\R} S_{h_{\bx}}(\Omega)$. Since~$\mathbf{A}_{h_{\bx}} \in \R^{N_{\bx} \times N_{\bx}}$ and~$\bM_{h_{\bx}} \in \R^{N_{\bx} \times N_{\bx}}$ are real symmetric matrices, and~$\bM_{h_{\bx}}$ is positive definite, the eigenvalues~$\mu_{h_{\bx}}$ are real. Moreover, the corresponding discrete eigenfunctions~$\{\Phi_{j,h_{\bx}}\}_{j=1}^{N_{\bx}}$ form a basis of~$\Shx(\Omega)$ orthogonal with respect to the~$L^2(\Omega)$ inner product.
\item Consequently, the discrete solution to~\eqref{eq:2.6} can be written as
\begin{equation*}
    \Uhp = \sum_{j=1}^{N_{\bx}} \psi_{j, h_t}(t) \Phi_{j, h_{\bx}}(\bx),
\end{equation*}
for some coefficients~$\psi_{j, h_t} \in \Sht{p}{p-1}(0, T)$.
The space--time method~\eqref{eq:2.6} then reduces to solving a set of independent problems similar to those in~\eqref{eq:3.2}.
\end{itemize}
\eremk
\end{remark}
Before focusing on stability, we first investigate the uniqueness of the solution to~\eqref{eq:3.2}. Due to the limitations of the matrix approach described in Section~\ref{sec:3.2} below, it is currently not possible to show the solvability of the linear system stemming from~\eqref{eq:3.2}. Consequently, a discrete variational argument seems to be required.

At the \textit{continuous} level, the uniqueness of the solution to the variational formulation of  \eqref{eq:3.1} holds, as shown in the following result.
\begin{proposition} \label{prop:3.3}
Given~$f \in L^2(0, T)$, $\mu \in \R,$ and~$\psi_0 \in \C$, the following variational problem has a unique solution: find~$\psi \in \psi_0 + H^1_{0,\bullet}(0,T)$ such that
\begin{equation} \label{eq:3.3}
        \mi (\dpt \psi, w)_{J_T} + \mu (\psi, w)_{J_T} = (f, w)_{J_T} \quad \text{for all~} w \in L^2(0,T).
\end{equation}
\end{proposition}
\begin{proof}
We can show that
\begin{equation*}
\psi(t) = e^{-\mi \mu t} \psi_0 + e^{-\mi \mu t} \int_0^t e^{\mi \mu s} f(s) \dd s
\end{equation*}
solves~\eqref{eq:3.1} and, due to the regularity assumption on~$f$, it is also a solution to~\eqref{eq:3.3}.

To show uniqueness, suppose that $\psi \in H^1_{0,\bullet}(0,T)$ satisfies the homogeneous problem
\begin{equation} \label{eq:3.4}
        \mi (\dpt \psi, w)_{J_T} + \mu (\psi, w)_{J_T} = 0 \quad \text{for all~} w \in L^2(0,T).
\end{equation}
We test~\eqref{eq:3.4} with~$w(t) = e^{-t/T} \dpt \psi(t)$. It can be easily seen that
\begin{equation} \label{eq:3.5}
    (\dpt \psi, e^{-\bullet/T} \dpt \psi)_{J_T} = \int_0^T \partial_s \psi(s) \partial_s \overline{\psi}(s) e^{-s/T} \dd s = \int_0^T |\partial_s \psi(s)|^2 e^{-s/T} \dd s.
\end{equation}
Using the identity 
\begin{equation} \label{eq:3.6}
    \text{Re}(\psi \dpt \overline{\psi}) = \frac{1}{2} \dpt |\psi|^2
\end{equation}
and integration by parts, we also obtain
\begin{equation} \label{eq:3.7}
    \text{Re}(\psi, e^{-\bullet/T} \dpt \psi)_{J_T} = \int_0^T \text{Re}(\psi(s) \partial_s \overline{\psi}(s)) e^{-s/T} \dd s = \frac{1}{2T} \int_0^T |\psi(s)|^2 e^{-s/T} \dd s + \frac{1}{2e} |\psi(T)|^2.
\end{equation}
Inserting \eqref{eq:3.5} and \eqref{eq:3.7} in \eqref{eq:3.4}, we conclude  
\begin{equation*} 
    0 = \text{Re}(\mi (\dpt \psi, e^{-\bullet/T} \dpt \psi)_{J_T} + \mu (\psi, e^{-\bullet/T} \dpt \psi)_{J_T}) = 
     \frac{\mu}{2T} \int_0^T |\psi(s)|^2 e^{-s/T} \dd s +  \frac{\mu}{2e} |\psi(T)|^2,
\end{equation*}
from which we deduce $\psi \equiv 0$.
\end{proof}
Unfortunately, the argument used in Proposition~\ref{prop:3.3} does not extend to the \emph{discrete} level, as the test function~$e^{-t/T} \dpt \psi_h^p(t)$ does not belong to the discrete space. One could adapt the scheme~\eqref{eq:3.2} by incorporating exponential functions into the test space (or, equivalently, modifying the
sesquilinear form) in the spirit of the methods in~\cite{FerrariPerugia2026, FerrariPerugiaZampa2025}. 
However, this approach lies beyond the scope of the present paper, as we are primarily interested in the conservation properties reported in Proposition~\ref{prop:2.3}, which would not be satisfied by the modified formulation.
\begin{remark}[Uniqueness at the discrete level for small~$\mu$]
At the discrete level, we can show that, if~$|\mu|$ is sufficiently small, then problem \eqref{eq:3.2} has a unique solution. Indeed, let $\psi_{h_{t}}^p \in S_{h_{t},0,\bullet}^p(0,T,\C)$ be a solution to~\eqref{eq:3.2} with zero data, i.e., it satisfies (see Remark \ref{rem:31})
\begin{equation} \label{eq:3.8} 
        \mi (\dpt \psi_{h_{t}}^p, \dpt w_{h_{t}}^p)_{J_T} - \mu (\dpt \psi_{h_{t}}^p, w_{h_{t}}^p)_{J_T} = 0 \quad \text{for all~} w_{h_{t}}^p \in S_{h_{t},\bullet,0}^p(0,T,\C).
\end{equation}
Choosing the test function $w_{h_{t}}^p = \psi_{h_{t}}^p-\psi_{h_{t}}^p(T)$ in \eqref{eq:3.8}, we obtain
\begin{equation} \label{eq:3.9}
    0=\mi \| \dpt \psi_{h_{t}}^p \|^2_{L^2(0,T)} - \mu (\dpt \psi_{h_{t}}^p, \psi_{h_{t}}^p - \psi_{h_{t}}^p(T))_{J_T} = \mi \| \dpt \psi_{h_{t}}^p \|^2_{L^2(0,T)} - \mu (\dpt \psi_{h_{t}}^p, \psi_{h_{t}}^p)_{J_T} + \mu | \psi_{h_{t}}^p(T)|^2.
\end{equation}
We split~\eqref{eq:3.9} into its real and imaginary parts. Taking the real part of~\eqref{eq:3.9} and using  
identity \eqref{eq:3.6}, we get
\begin{equation*}
    | \psi_{h_{t}}^p(T) |^2 = \Re (\dpt \psi_{h_{t}}^p, \psi_{h_{t}}^p)_{J_T} = \frac{1}{2} |\psi_{h_{t}}^p(T)|^2, 
    \quad \text{whence~} \, \psi_{h_{t}}^p(T)=0.
\end{equation*}
Moreover, taking the imaginary part of \eqref{eq:3.9} gives
\begin{equation*}
    \| \dpt \psi_{h_{t}}^p \|^2_{L^2(0,T)} = \mu \Im(\dpt \psi_{h_{t}}^p, \psi_{h_{t}}^p)_{J_T}. 
\end{equation*}
Then, using the Cauchy--Schwarz inequality, and the Poincar\'e inequality in one dimension with the sharp constant~$T/\pi$ (which holds for functions vanishing at the endpoints), we obtain
\begin{equation*}
    \| \dpt \psi_{h_{t}}^p \|^2_{L^2(0,T)} \le |\mu| \| \dpt \psi_{h_{t}}^p \|_{L^2(0,T)} \| \psi_{h_{t}}^p\|_{L^2(0,T)} \le |\mu| \frac{T}{\pi} \| \dpt \psi_{h_{t}}^p \|^2_{L^2(0,T)}.
\end{equation*}
If~$|\mu| < T/\pi$, this implies that~$\| \partial_t \psi_{h_{t}}^p\|_{L^2(0,T)}=0$, and thus~$\psi_{h_{t}}^p \equiv 0$.

Although this result establishes a range of values for $\mu$ guaranteeing uniqueness, this theoretical bound is far from being sharp. In practice, we have never encountered a situation 
where uniqueness fails at the discrete level (see also Remark~\ref{rem:3.7} below).
\eremk
\end{remark}
Since a rigorous proof of uniqueness for formulation~\eqref{eq:3.2} using variational tools currently seems out of reach, we rely on the following fundamental assumption for the rest of this section.
\begin{assumption} \label{assu:3.5}
We assume that $p \in \mathbb{N}$ is such that, for all $h_{t} > 0$ and $\mu \in \R$, the variational formulation~\eqref{eq:3.2} admits a unique solution.
\end{assumption}
\begin{remark}
Assumption \ref{assu:3.5} holds for $p=1$. Indeed, when employing standard hat functions, the system matrix stemming from~\eqref{eq:3.2} is a lower triangular matrix with entries all equal to~$-\left(8\mi/h_{t} +\mu/2\right)$ on the diagonal. Its invertibility for all~$h_{t}>0$ and~$\mu \in \R$ readily follows. Furthermore, for this specific case, the proposed scheme coincides with the method introduced in  \cite{KarakashianMakridakis1999}. For~$p>1$, the larger support of B-splines does not provide sufficient insight to establish the invertibility of the system matrix through a linear algebra argument. We refer to Remark \ref{rem:3.7} below for a numerical verification.
\eremk
\end{remark}

\subsection{Toeplitz structure of the system matrix and weakly well-conditioning} \label{sec:3.2}

In this section, under the crucial Assumption \ref{assu:3.5} of uniqueness of the solution to the discrete problem, we show unconditional stability of the scheme~\eqref{eq:3.2} via a  matrix-based approach, in the same spirit as
in~\cite{FerrariFraschiniLoliPerugia2025, FerrariFraschini2026}. We first define the matrices involved in the discrete formulation.

Let consider the B-splines $\{ \phi_j^p\}_{j=0}^{N_{t}+p-1}$ of degree~$p$ with maximal regularity~$C^{p-1}$ defined according to the Cox-De Boor recursion formula~\cite[\S1]{deBoor1972}. Specifically, we consider the \emph{knot} vector composed of~$p+1$ repetitions of~$t_0=0$, multiplicity one of each $t_1, \ldots, t_{N_t-1}$, and $p+1$ repetitions of $t_{N_t}=T$. 
Let~$\{\xi^p_j\}_{j=0}^{N_t+2p}$ denote this knot vector. The B-splines with maximal regularity $p-1$ are defined recursively in $k$ (with the convention that rational functions with zero denominator are identically zero) as
{\small\begin{equation*} 
    \varphi_j^{k}(t) = 
    \begin{cases}
        \displaystyle \frac{t-\xi_j^{p}}{\xi^{p}_{j+k}-\xi^{p}_j} \varphi_j^{k-1}(t) + \displaystyle \frac{\xi^{p}_{j+k+1}-t}{\xi^{p}_{j+k+1}-\xi^{p}_{j+1}} \varphi_{j+1}^{k-1}(t),
        & \text{if } t \in [\xi_j^p,\xi^p_{j+k+1}), \\
        0, & \text{otherwise},
    \end{cases}
\end{equation*}
for $j=0,\ldots,N_t+p-1$, with $\varphi_j^{0}(t) = 1$ if $t \in [\xi_j^p,\xi_{j+1}^p)$, and $\varphi_j^0(t) = 0$ otherwise.} 
In particular, the basis is ordered such that, for the spaces defined in \eqref{eq:2.5}, the following representations hold:
\begin{equation*}
    S_{h_{t},0,\bullet}^{p}(0,T,\C) = \text{span}_{\C}\{\phi_j^p : j = 1,\ldots, N_{t}+p-1\}, \quad S_{h_{t},\bullet,0}^{p}(0,T,\C) = \text{span}_{\C}\{\phi_j^p : j = 0,\ldots, N_{t}+p-2\}. 
\end{equation*}
Hence, the matrices involved in the discrete variational formulation~\eqref{eq:3.2} are defined as
\begin{equation} \label{eq:3.10}
	\B^p_{h_{t}}[\ell,j] := (\dpt \phi_j^p, \dpt \phi^p_{\ell-1})_{J_T}, \quad \bC^p_{h_{t}}[\ell,j] := (\dpt \phi_j^p, \phi^p_{\ell-1})_{J_T},
\end{equation}
for $\ell, j = 1,\ldots,N_{t}+p-1$. The system matrix associated with formulation \eqref{eq:3.2} then reads (see Remark~\ref{rem:31})
\begin{equation} \label{eq:3.11}
    \mathbf{K}^p_{h_t}(\mu) := \mi \B^p_{h_t} - \mu \bC^p_{h_t}.
\end{equation}
\begin{remark}[Assumption \ref{assu:3.5} as a generalized eigenvalue problem] \label{rem:3.7}
Assumption \ref{assu:3.5} is equivalent to the invertibility of the matrix defined in \eqref{eq:3.11} for all $\mu \in \R$ and $h_{t} > 0$. For a given meshsize $h_{t}$, this is equivalent to stating that the corresponding generalized eigenvalue problem
\begin{equation} \label{eq:3.12}
    \text{find~} (\lambda,\mathbf{v}) \in \C \times \C^{N_t+p-1} \quad \text{~such that~} \quad \mathbf{B}^p_{h_{t}} \mathbf{v} = \lambda \mathbf{C}^p_{h_{t}} \mathbf{v},
\end{equation}
does not admit any purely imaginary eigenvalues.

In Figure~\ref{fig:1}, we report the numerical approximation of the eigenvalues for the generalized eigenvalue problem~\eqref{eq:3.12}, computed using the \texttt{eig} function in MATLAB R2025b.\footnote{Due to the ill-conditioning associated with the computation of these eigenvalues, this numerical verification requires evaluating the matrix entries with high machine precision. Consequently, the matrices were generated using the MATLAB's Variable Precision Arithmetic (\texttt{vpa}), setting a precision of $1000$ decimal digits. The source code used for these computations, along with the high-precision matrices, is available for verification in the \texttt{verifications} folder of the repository~\cite{FerrariCodes2025}.} In the reported experiments, for a fixed degree $p$, all eigenvalues are uniformly bounded away from the imaginary axis within the left complex half-plane. In particular, none of the computed eigenvalues are purely imaginary, providing empirical evidence in support of Assumption~\ref{assu:3.5}. \eremk

\begin{figure}[ht]
        \centering
        \includegraphics[width=0.3\linewidth]{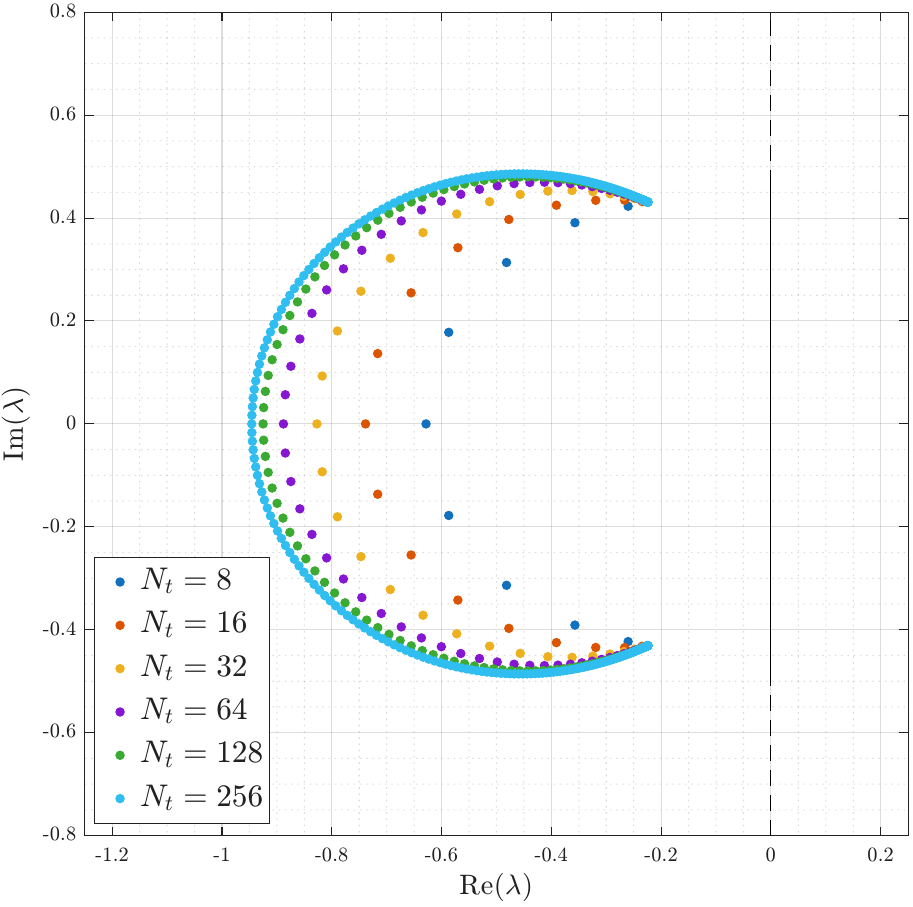}
        \includegraphics[width=0.3\linewidth]{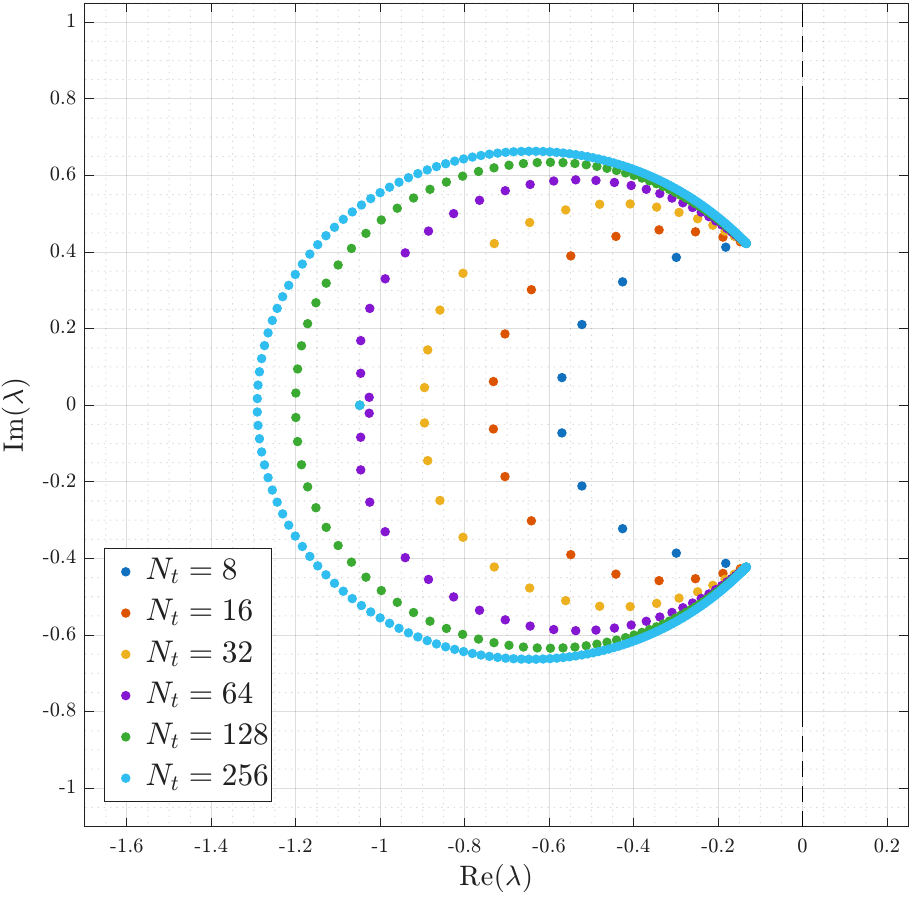} \\
        \includegraphics[width=0.3\linewidth]{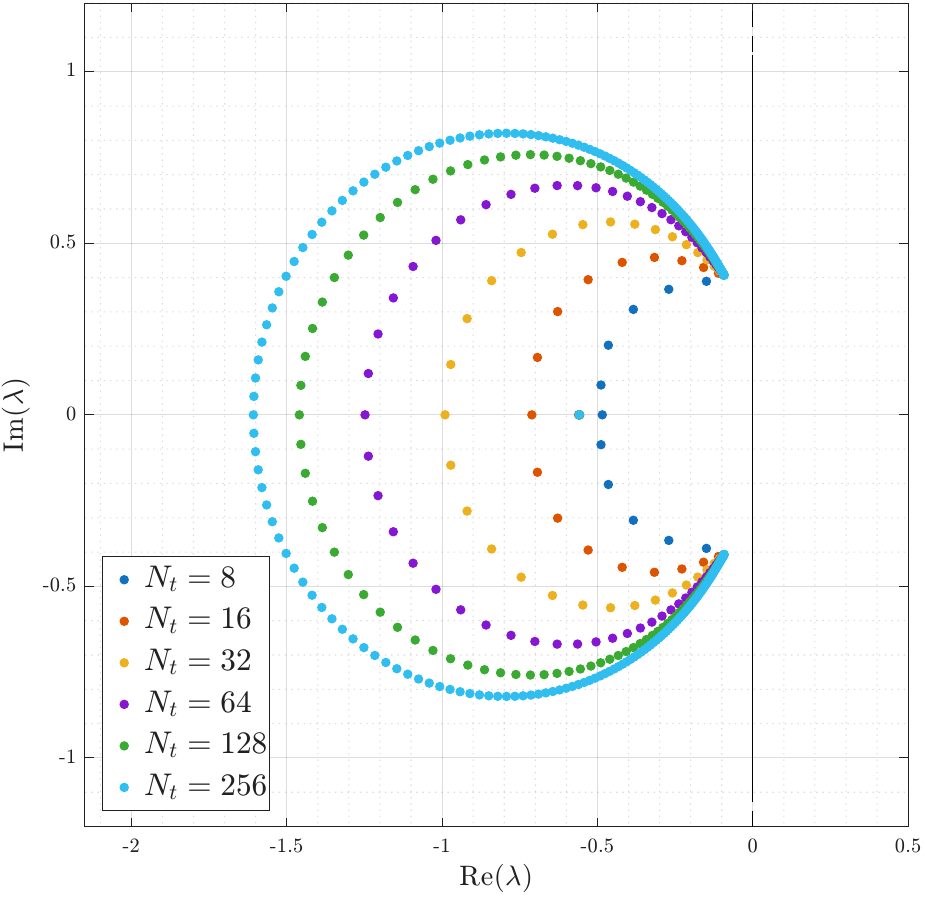}
        \includegraphics[width=0.3\linewidth]{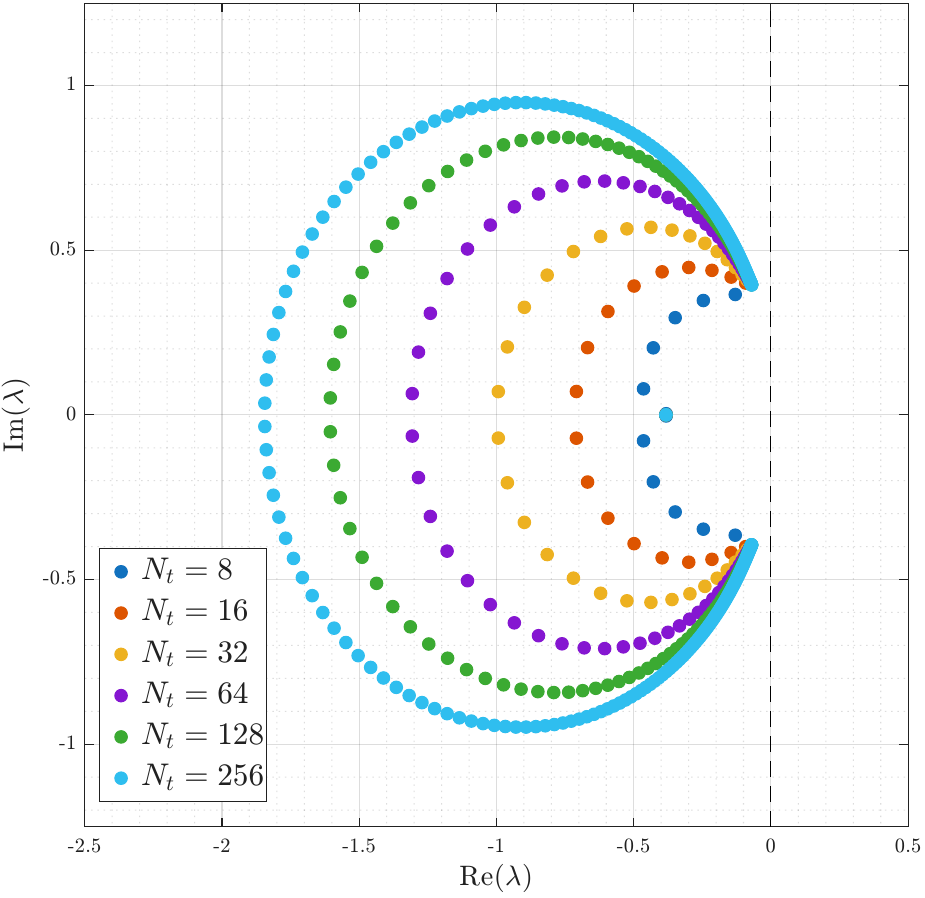}
        \caption{Numerical approximation of the eigenvalues for the generalized eigenvalue problem~\eqref{eq:3.12}. The results are shown for spline degrees~$p =2$ (top-left), $p=3$ (top-right), $p=4$ (bottom-left), and $p=5$ (bottom-right).}
        \label{fig:1}
\end{figure}
\end{remark}
The aim of this section is to study the behavior of the condition number of the family of matrices $\{\K_{h_{t}}^p(\mu)\}$ with respect to $h_{t}$, $p$, and $\mu$. In particular, we show that, for all $p \in \N$ and $\mu \in \R$, the condition number of~$\K_{h_{t}}^p(\mu)$, as a function of~$h_{t}$, does not grow exponentially but rather as an algebraic power of~$1/h_{t}$. This is the type of unconditional stability that we aim for, which contrasts with conditional stability, where a particular relation between~$h_{t}$, $p$, and $\mu$ is required so that the condition number grows as an algebraic power of~$1/h_{t}$, and it grows exponentially otherwise.

The subsequent analysis is based on four main steps, which are common also to the analyses carried out in~\cite{FerrariFraschini2026} and~\cite{FerrariFraschiniLoliPerugia2025}. The outline is as follows.
\begin{enumerate}
    \item We show that the entries of the matrices~$ h_{t}\B_{h_{t}}^p $ and $ \bC_{h_{t}}^p $, defined in \eqref{eq:3.10}, do not depend on $h_{t}$, and that they are Toeplitz matrices with specific symmetries, except for a fixed number of entries (depending only on $p$) in the top-left and bottom-right corners. We refer to such matrices as \textit{nearly Toeplitz}. As a consequence, the system matrix $\mathbf{K}_{h_{t}}^p(\mu)$ is nearly Toeplitz and exhibits the same structure, and up to a rescaling, its entries only depend on the quantity~$ \mu h_{t}$. 
    \item  We recall the definition in~\cite{AmodioBrugnano1996} of \textit{weakly well-conditioning} for a family of matrices, and state a sufficient condition for the weakly well-conditioning of families of banded Toeplitz matrices~\cite[Thm.~3]{AmodioBrugnano1996}, along with its generalization to nearly Toeplitz matrices~\cite[Thm.~4.5]{FerrariFraschini2026}. This sufficient condition involves characterizing the zeros of the \textit{symbol polynomial} associated with the family.
    \item Given the symmetries of $h_t\B^p_{h_t}$ and $\bC^p_{h_t}$ exploited in the first step, we obtain that, if the symbol polynomial associated with the family $\{\K_{h_t}^p(\mu)\}$ has exactly two zeros of unitary modulus, then $\{\K_{h_t}^p(\mu)\}$ is weakly well-conditioned.
    \item Finally, we compute the symbol polynomial associated with the family $ \{ \K_{h_t}^p(\mu) \}$ and show that it has exactly two zeros of unitary modulus for all $ \mu \in \R$ and $ p \in \N$.
\end{enumerate}

\paragraph{Step 1. Nearly Toeplitz structure of the matrices $h_t\B_{h_t}^p$ and $\bC_{h_t}^p$} 
\label{step:1}
\phantom{of the opera} \\
The next results have been proven in~\cite[Prop.~3.2]{FerrariFraschini2026} and~\cite[Prop.~2.3]{FerrariFraschiniLoliPerugia2025} (see also  \cite[\S 4.1]{GaroniManniPelosiSerraCapizzanoSpeleers2014}).
\begin{proposition} \label{prop:3.8}
Let~$p \in \N$ and let the matrices~$\B_{h_t}^p$ and~$\bC_{h_t}^p$ be defined in~\eqref{eq:3.10}. Then, the following properties hold true.
\begin{enumerate}[label=\arabic*), ]
\item \label{P1} The entries of the matrices~$h_t\B_{h_t}^p$ and~$\bC_{h_t}^p$ are independent of the mesh parameter~$h_t$.
\item The matrices~$\B_{h_t}^p$ and~$\bC_{h_t}^p$ are persymmetric, i.e., they are symmetric about their northeast-to-southwest diagonal (anti-diagonal).
\item \label{P3} The matrices $\B_{h_t}^1$ and $\bC_{h_t}^1$ are lower triangular Toeplitz band matrices with three nonzero diagonals. For $p > 1$, the matrices $\B_{h_t}^p$ and $\bC_{h_t}^p$, except for $2(2p^2 - 3)$ entries located at the top-left and bottom-right corners, exhibit a Toeplitz band structure with $p+1$ bands below the main diagonal and $p-1$ above. Moreover, in the purely Toeplitz band part, $\B_{h_t}^p$ and $\bC_{h_t}^p$ exhibit symmetry and skew-symmetry, respectively, relative to the first lower co-diagonal. In the top-left corner, the non-zero entries that do not follow the Toeplitz structure are the following:
\begin{equation*} 
    \vphantom{
    \begin{matrix}
    \overbrace{XYZ}^{\mbox{$R$}}\\ \\ \\ \\ \\ \\ \\ \\
    \underbrace{pqr}_{\mbox{$S$}} \\
    \end{matrix}}
    \begin{matrix}
        \vphantom{a}
        \coolleftbrace{p+1}{* \\ \vdots \\ * \\ * \vspace{0.5cm}} \vspace{0.1cm}\\ 
        \coolleftbrace{p-2}{* \\ \vdots \\ *}
    \end{matrix}
    \begin{pmatrix}
        \coolover{p-1}{*     & \ldots &     * \hspace{0.1cm}}& \coolover{p-1}{* & \phantom{***} & \phantom{*}  } 
        \\        \vdots &        & \vdots               &         \vdots & \ddots           & 
        \\        *      & \ldots & *                    & *              &  \ldots          &  *           
        \\         *     & \ldots &     *                &              * &  \ldots          &  *           &  
        \\        *      & \ldots & *                    &                &                  &              
        \\         *     & \ldots &     *                &                &                  &                            &
        \\               & \hspace{-0.2cm}\ddots & \vdots               &                &                   
        \\               &        &  \hspace{-0.3cm}*  \, \,\, *    &                &                   
    \end{pmatrix}.
\end{equation*}
A similar (persymmetric) structure holds for the bottom-right corner.
\end{enumerate}
\end{proposition}
\begin{example} \label{ex:3.9}
As an illustrative example, we write the entries of the matrices $\B_{h_t}^2, \bC_{h_t}^2 \in \R^{(N_t+1)\times (N_t+1)}$ to highlight their algebraic structure
\begin{align*}
h_t \B_{h_t}^2 &= \frac{1}{6} 
\setcounter{MaxMatrixCols}{20}
    \begin{pmatrix} 
            \colorbox{yellow!30}{$-6$} & \colorbox{yellow!30}{$-2$}  \\ 
            \colorbox{yellow!30}{$8$} & \colorbox{yellow!30}{$-1$} & -1 \\
            \colorbox{yellow!30}{$-1$} & 6 & -2 & -1 \\
            -1 & -2 & 6 & -2 & -1 \\
            & \ddots & \ddots & \ddots & \ddots & \ddots \\
            & & -1 & -2 & 6 & -2 & -1 \\
            & & & -1 & -2 & 6 & \colorbox{yellow!30}{$-1$} & \colorbox{yellow!30}{$-2$} \\
            & & & & -1 & \colorbox{yellow!30}{$-1$} & \colorbox{yellow!30}{$8$} & \colorbox{yellow!30}{$-6$} \\
    \end{pmatrix}, \\[10pt]
\bC_{h_t}^2 &= \frac{1}{24} 
\setcounter{MaxMatrixCols}{20}
    \begin{pmatrix} 
            \colorbox{yellow!30}{$10$} & \colorbox{yellow!30}{$2$}  \\
            \colorbox{yellow!30}{$0$} & \colorbox{yellow!30}{$9$} & 1 \\
            \colorbox{yellow!30}{$-9$} & 0 & 10 & 1 \\
            -1 & -10 & 0 & 10 & 1 \\
            & \ddots & \ddots & \ddots & \ddots & \ddots \\
            & &  -1 & -10 & 0 & 10 & 1 \\
            & & & -1 & -10 & 0 & \colorbox{yellow!30}{$9$} & \colorbox{yellow!30}{$2$} \\
            & & & & -1 & \colorbox{yellow!30}{$-9$} & \colorbox{yellow!30}{$0$} & \colorbox{yellow!30}{$10$} \\
    \end{pmatrix}.
\end{align*}
With the exception of $2\cdot 2^2-3 = 5$ entries in the top-left corner and $5$ entries in the bottom-right corner (highlighted in yellow), both are banded Toeplitz matrices. Furthermore, the matrix $\B_{h_t}^2$ exhibits symmetry with respect to its first lower codiagonal, while $\bC_{h_t}^2$ exhibits skew-symmetry.
\end{example}
Due to the scaling property~\ref{P1} in Proposition \ref{prop:3.8}, it is reasonable to introduce the families of matrices
\begin{equation} \label{eq:3.13}
    \{\B_n^p\}_n \quad \text{and} \quad \{\bC_n^p\}_n,
\end{equation}
with $\B_n^p := h_{t}\B_{h_{t}}^p$, $\bC_n^p := \bC_{h_{t}}^p$, and $n := N_t + p - 1$. Here and in the following sections, the subscript~$n$ denotes the size of the involved matrix. We also define the following scaled version of the system matrix in \eqref{eq:3.11}:
\begin{equation} \label{eq:3.14}
    \K_n^p(\rho) := \mi \B_n^p - \rho \bC_n^p,
\end{equation}
where we have set~$\rho := \mu h_t$. Thus, the relation~$\mathbf{K}_n^p(\rho) = h^{-1}_t \mathbf{K}_{h_t}^p(\mu)$ holds. Consequently, the entries of the scaled matrix $\mathbf{K}_n^p(\rho)$ depend on the original parameters $h_t$ and $\mu$ only through their product $\rho$.

We say that a family of matrices is composed of \emph{nearly Toeplitz} matrices if each matrix in the family is a perturbation of a Toeplitz matrix, where the perturbation matrix has a number of nonzero entries that is independent of $n$ and located only in the top-left and/or bottom-right corners. From Proposition \ref{prop:3.8}, the families $ \{\B_n^p\}_n $ and $ \{\bC_n^p\}_n $ consist of nearly Toeplitz matrices. Consequently, the family $ \{\K_n^p(\rho)\}_n $ also consists of nearly Toeplitz matrices.

\paragraph{Step 2. Weakly well-conditioning of nearly Toeplitz matrices.} \label{step:2} 
\phantom{12}\\ 
A family of nonsingular matrices $ \{\A_n\}_n $ with $ \A_n \in \C^{n \times n} $ is said to be \textit{well conditioned} if the condition numbers $ \kappa(\A_n) $ are uniformly bounded with respect to~$ n $. It is said to be~\textit{weakly well-conditioned} if~$ \kappa(\A_n) $ grows as a power of~$ n $. Clearly, these definitions do not depend on the chosen matrix norm.

Given~$m,k \in \N$ and coefficients $\{a_j\}_{j=-m}^k \subset \C$, we consider the family of Toeplitz band matrices $\{ \A_n \}_{n}$, with structure
\begin{equation} \label{eq:3.15}
\A_n =  
\setcounter{MaxMatrixCols}{20}
    \begin{pmatrix} 
            a_0 & \ldots & a_k & &  \\ 
            \vdots & \ddots & & \ddots \\
            a_{-m} & & \ddots & &  a_k \\
            & \ddots & & \ddots & \vdots \\
            &  & a_{-m} & \ldots & a_0 \\
    \end{pmatrix}_{n \times n}.
\end{equation}
We associate with the family $\{\A_n\}_n$, the \textit{symbol polynomial} $q^{\A} \in \P_{m+k}(\C)$ defined as
\begin{equation} \label{eq:3.16}
    q^{\A}(z) := \sum_{j=-m}^k a_j z^{m+j}.
\end{equation}
In \cite[Thm.~3]{AmodioBrugnano1996}, a sufficient condition for the weakly well-conditioning of a family of Toeplitz band matrices has been established, which involves locating the zeros of its symbol polynomial. We say that a polynomial 
is of type~$(s,u,\ell)$ if it has $s$ zeros with modulus smaller than $1$, $u$ zeros with unit modulus, and $\ell$ zeros with modulus larger than $1$.
\begin{theorem}[{see~\cite[Thm.~3]{AmodioBrugnano1996}}] \label{th:3.10}
Let $\{\A_n\}_{n}$ represent a family of nonsingular Toeplitz band matrices as in \eqref{eq:3.15}, and let~$q^{\A}$ in \eqref{eq:3.16} be its symbol polynomial. Then, the family $\{\A_n\}_{n}$
\begin{itemize}
\item[a)] is well conditioned if $q^{\A}$ is of type $(m,0,k)$;
\item[b)] is weakly well-conditioned if $q^{\A}$ is of type $(m_1,m_2,k)$ or $(m,k_1,k_2)$, where $m_1+m_2=m$ and $k_1+k_2=k$.
\end{itemize}
\end{theorem}
Since the family of matrices $ \{\K_n^p(\rho)\}_n $ defined in \eqref{eq:3.14} consists of nearly Toeplitz matrices rather than Toeplitz matrices, we cannot directly apply Theorem \ref{th:3.10}. Instead, we use an auxiliary result derived in \cite{FerrariFraschini2026}.
\begin{theorem}[{see~\cite[Thm.~4.5]{FerrariFraschini2026}}] \label{th:3.11}
Let $ \{\A_n\}_n $ be a family of nonsingular Toeplitz band matrices as in \eqref{eq:3.15}. Consider a family $ \{\widehat{\A}_n\}_n $ consisting of nonsingular nearly Toeplitz matrices that are perturbations of $ \{\A_n\}_n $, where the perturbations of each $ \widehat{\A}_n $ are only in the top-left and bottom-right blocks of size $(m+k) \times (m+k)$, with entries that are independent of $n$, i.e., the perturbation blocks are of the form
\begin{equation} \label{eq:3.17}
    \vphantom{
    \begin{matrix}
            \overbrace{XYZ}^{\mbox{$R$}} \\ \\ \\ \\ \\
            \underbrace{pqr}_{\mbox{$S$}} \\
        \end{matrix}}
    \begin{matrix}
        \coolleftbrace{m}{* \\ \vspace{0.15cm} \\ * } \vspace{0.1cm}\\ 
        \coolleftbrace{k}{* \\ \vspace{0.15cm} \\ *}
    \end{matrix}%
    \hspace{-0.15cm}
    \begin{pmatrix}
        \coolover{k}{*     & \ldots &     * \hspace{0.2cm}}& \coolover{m}{\squareasterisk & \phantom{**}  & \phantom{**}} 
        \\        \vdots &        & \vdots               &         \vdots & \ddots           & 
        \\         *     & \ldots &     *                &              * &  \ldots          &  \squareasterisk           
        \\         \squareasterisk     & \ldots &     *                &                &                  &                            
        \\               & \ddots & \vdots               &                &                  & 
        \\               &        & \squareasterisk                    &                &                  &  
        \end{pmatrix}, \quad\quad
         \vphantom{
    \begin{matrix}
            \overbrace{XYZ}^{\mbox{$R$}}\\ \\ \\ \\ \\ \\ \\ \\
            \underbrace{pqr}_{\mbox{$S$}} \\
        \end{matrix}}%
    \begin{pmatrix}
         &        &                &        \squareasterisk  &            & 
        \\              &  &                     &   \vdots            &  \ddots          &             
        \\              &                &                  & *     & \ldots &     \squareasterisk                             
        \\             \squareasterisk  & \ldots & *   &       *     &    \ldots            &                 *  
        \\               &  \ddots & \vdots                    &   \vdots             &              &  \vdots
        \\ \coolunder{k}{ \phantom{**}  & \phantom{**} &     \squareasterisk } & \coolunder{m}{* & \ldots & *} 
        \end{pmatrix}
        \hspace{-0.15cm}
    \begin{matrix}
        \coolrightbrace{* \\ \vspace{0.15cm} \\ * }{m} \vspace{0.1cm}\\ 
        \coolrightbrace{* \\ \vspace{0.15cm} \\ * }{k} \vspace{0.1cm}
    \end{matrix}%
\end{equation}
where the matrix on the left represents the top-left perturbation block, and the matrix on the right represents the bottom-right perturbation block. Furthermore, assume that the matrices $ \widehat{\A}_n $ have nonzero entries in the two outer codiagonals (the~$k^{\thh}$ and the~$(-m)^{\thh}$), i.e., the entries in the square boxes of the two perturbed blocks in \eqref{eq:3.17}. Under these conditions, the families $ \{\A_n\}_n $ and $ \{\widehat{\A}_n\}_n $ exhibit the same conditioning behavior.
\end{theorem}
Theorem \ref{th:3.11} is applicable to $ \{\K_n^p(\rho)\}_n$ with~$m=p+1$ and~$k=p-1$ (see Proposition \ref{prop:3.8}, point \ref{P3}), and to the family consisting of the extensions of the purely Toeplitz band parts of $ \K_n^p(\rho) $ to $ n \times n $ matrices, provided that we show the entries of the two outer codiagonals (the $(-p-1)^{\text{th}}$ and the $(p-1)^{\text{th}}$) of $\K_n^p(\rho)$ are nonzero.
\begin{proposition} \label{prop:3.12}
For all $p \in \N$ and $\rho \in \R$, we have, for all $j=1,\ldots,n-p+1$,
\begin{equation} \label{eq:3.18} 
    \K^p_n(\rho)[j,j+p-1] \ne 0, \quad\quad 
    \K^p_n(\rho)[j+p+1,j] \ne 0
\end{equation}
In particular, the hypotheses of Theorem \ref{th:3.11} are satisfied by the family of matrices $\{\K_n^p(\rho)\}_n$.
\end{proposition}
\begin{proof}
In \cite[Prop.~3.6]{FerrariFraschiniLoliPerugia2025}, it was shown that, for all $p \in \N$, it holds true that
\begin{equation*}
    \bC^p_n[j,j+p-1] > 0 \quad\text{and}\quad  \B^p_n[j,j+p-1] < 0, \quad \text{for } j=1,\ldots,n-p+1.
\end{equation*}
Since $\bC^p_n, \B_n^p \in \R^{n \times n}$ and $\rho \in \R$, we deduce from definition \eqref{eq:3.14} that
\begin{equation*}
    \Re \bigl(\K^p_n(\rho)[j,j+p-1]\bigr) = -\rho \bC^p_n[j,j+p-1] \ne 0, \quad\quad 
    \Im \bigl(\K^p_n(\rho)[j,j+p-1]\bigr) = \B^p_n[j,j+p-1] < 0.
\end{equation*}
This yields the first assertion of~\eqref{eq:3.18}.
Using an analogous argument, the second assertion follows.
\end{proof}
\paragraph{Step 3. Reciprocity of the symbol polynomial \texorpdfstring{$q^{\K^p(\rho)}$}{1}} \label{step:3} \phantom{123}\\
By virtue of Theorem~\ref{th:3.11}, to study the stability of problem \eqref{eq:3.2}, it suffices to study the conditioning of the family of extensions of the purely Toeplitz band parts of~$ \K_n^p(\rho) $. Without ambiguity, we associate the symbol polynomial $ q^{\K^p(\rho)} $ of this family with the family of nearly Toeplitz matrices $ \{\K_n^p(\rho)\}_n $. We proceed similarly for $ \{\B_n^p\}_n $ and $ \{\bC_n^p\}_n $, denoting their symbol polynomials by~$ q^{\B^p} $ and $ q^{\bC^p} $, respectively. From property~\ref{P3} in Proposition \ref{prop:3.8}, the polynomials $ q^{\B^p} $ and $ q^{\bC^p} $ have degree $2p$, and their coefficients depend only on~$p$. By linearity of the symbol polynomials, we deduce
\begin{equation} \label{eq:3.19}
    q^{\K^p(\rho)}(z) = \mi q^{\B^p}(z) - \rho q^{\bC^p}(z).
\end{equation}
\begin{example}
Following up Example~\ref{ex:3.9}, we can explicitly write the symbol polynomials for $p=2$ by directly applying definition \eqref{eq:3.16}. Since the purely Toeplitz parts have $m=p+1=3$ lower bands and $k=p-1=1$ upper band, we have%
\begin{equation} \label{eq:3.20}
\begin{aligned}
    q^{\mathbf{B}^2}(z) = \frac{1}{6} \left( -1 - 2z + 6z^2 - 2z^3 - z^4 \right), \quad
    q^{\mathbf{C}^2}(z) = \frac{1}{24} \left( -1 - 10z + 10z^3 + z^4 \right).
\end{aligned}
\end{equation}
\end{example}
In the next proposition, we show that $q^{\K^p(\rho)}$ is a \textit{reciprocal} polynomial, meaning that, if $\xi \in \mathbb{C}$ is a root of~$q^{\K^p(\rho)}$, then $\xi^{-1}$ is also a root. Using this information, we then prove that, if~$q^{\K^p(\rho)}$ has two zeros with unitary modulus, then the family $\{ \K_n^p(\rho) \}_n$ is weakly well-conditioned.
\begin{proposition} \label{prop:3.14}
The symbol polynomial $q^{\mathbf{K}^p(\rho)}$ of the family $\{ \K_n^p(\rho) \}_n$ defined in \eqref{eq:3.14} is a reciprocal polynomial. As a consequence, if the polynomial $q^{\K^p(\rho)}$ has two zeros of unit modulus and %
Assumption \ref{assu:3.5} is satisfied, then the family $\{ \K_n^p(\rho) \}_n$ is weakly well-conditioned.
\end{proposition}
\begin{proof}
From the symmetry and skew-symmetry of the matrices $\B_n^p$ and $\bC_n^p$ (see property~\ref{P3} in Proposition~\ref{prop:3.8}), the polynomials $q^{\B^p}$ and $q^{\bC^p}$ can be expressed as
\begin{align*}
    q^{\B^p}(z) & = b_0 + b_1 z + \ldots + b_p z^p + \ldots + b_1 z^{2p-1} + b_0 z^{2p} = z^p \sum_{j=0}^p b_j \left( z^{-(p+j)} + z^{p+j} \right),
    \\ q^{\bC^p}(z) & = c_0 + c_1 z + \ldots + c_{p-1} z^{p-1} - c_{p-1} z^{p+1} - \ldots - c_1 z^{2p-1} - c_0 z^{2p}= z^p \sum_{j=0}^p c_j \left( z^{-(p+j)} - z^{p+j} \right),
\end{align*}
for some coefficients $\{b_j\}_j, \{c_j\}_j \subset \R$, with $b_0 \neq 0$ and $c_0 \neq 0$ (see also Section \ref{step:2}). In particular, from these expressions, we deduce that both $q^{\B^p}$ and $q^{\bC^p}$ are reciprocal polynomials, and from \eqref{eq:3.19}, it follows that $q^{\K^p(\rho)}$ is also a reciprocal polynomial for all $\rho \in \R$. As a consequence, $q^{\K^p(\rho)}$ is a polynomial of type $(\ell, 2p - 2\ell, \ell)$ for some $\ell = 0, \dots, p$. In the setting of Theorem~\ref{th:3.10}, for the matrices $\K^p_n(\rho)$, it holds that~$k = p - 1$ and~$m = p + 1$. Therefore, we can guarantee weakly well-conditioning (see also Theorem \ref{th:3.11}) if $\ell = p - 1$, resulting in exactly $2p - 2\ell = 2$ zeros of unit modulus for $q^{\K^p(\rho)}$.
\end{proof}
In the next section, we explicitly compute the restriction of $q^{\K^p(\rho)}$ to the boundary of the unit circle and show that, for all $\rho \in \R$, it has two zeros of unit modulus.

\paragraph{Step 4. Number of zeros of unitary modulus of the symbol polynomial~$q^{\K^p(\rho)}$} \label{step:4} \phantom{}\\
The aim of this section is to derive an explicit expression for $q^{\K^p(\rho)}(e^{\mi \theta})$ with~$\theta \in [-\pi,\pi]$ and subsequently show that, for all $p \in \mathbb{N}$ and $\rho \in \R$, there exist exactly two values of $\theta$, both with multiplicity one, such that $q^{\K^p(\rho)}(e^{\mi \theta}) = 0$. These simple zeros, expressed in polar coordinates, correspond bijectively to the simple complex zeros of $q^{\K^p(\rho)}$, as we establish in the following lemma.
\begin{lemma}[Characterization of simple zeros]
Let $f : \mathbb{C} \to \mathbb{C}$ be a complex function. Then, $z = e^{\mi \theta}$, with~$\theta \in [-\pi,\pi]$, is a simple zero of $f$ if and only if $\theta$ is a simple zero of the function $F(\theta) := f(e^{\mi \theta})$.
\end{lemma}
\begin{proof}
From the chain rule, we have
\begin{equation*}
    F'(\theta) = \frac{\text{d}}{\text{d}\theta} f(e^{\mi \theta}) = \mi e^{\mi \theta} f'(e^{\mi \theta}).
\end{equation*}
Therefore, $F'(\theta) \neq 0$ if and only if $f'(e^{\mi \theta}) \neq 0$, which shows that $z = e^{\mi \theta}$ is a simple zero of $f$ if and only if $\theta$ is a simple zero of $F$.
\end{proof}
In the next proposition, we recall the explicit expressions for the symbol polynomials $q^{\B^p}$ and $q^{\bC^p}$ established in \cite[Lemma 7]{GaroniManniPelosiSerraCapizzanoSpeleers2014} and~\cite[Lemma 3.3 b)]{Donatelli2016} (see also~\cite[Prop.~3.7]{FerrariFraschiniLoliPerugia2025}).
\begin{proposition} \label{prop:311}
For~$p \in \N$ and~$\theta \in [-\pi,\pi]$, define the functions ~$B_p, C_p : [-\pi,\pi] \to \R$ as
\begin{equation} \label{eq:3.21}
    B_p(\theta) := -(2-2\cos \theta)^{p+1} \sum_{j \in \Z} \frac{1}{(\theta + 2j\pi)^{2p}}, \quad C_p(\theta) := -(2-2\cos \theta)^{p+1} \sum_{j \in \Z} \frac{1}{(\theta + 2j\pi)^{2p+1}}.
\end{equation}
Then, the symbol polynomials $q^{\B^p}$ and $q^{\bC^p}$ associated with the families in \eqref{eq:3.13} are given by
\begin{equation} \label{eq:3.22}
     q^{\mathbf{B}^p}(e^{\mi \theta}) = - e^{\mi p \theta}  B_p(\theta), \quad\quad q^{\mathbf{C}^p}(e^{\mi \theta}) = - e^{\mi p \theta}\mi C_p(\theta).
\end{equation}
\end{proposition}
\begin{remark}[Equivalent expressions for~$B_p(\theta)$ and~$C_p(\theta)$]
 The infinite sums in \eqref{eq:3.21} can be expressed for all $p \in \N$ as rational functions of trigonometric functions. Indeed, starting from  (see, e.g., \cite[Eq.~1.422.4]{gradshteyn2007table})
\begin{equation*}
    \frac{1}{\sin^2(x)}= \sum_{j \in \Z} \frac{1}{(x+j\pi)^2},
\end{equation*}
and applying the substitution $x = \theta/2$, we obtain
\begin{equation*}
    \sum_{j \in \Z} \frac{1}{(\theta+2j\pi)^2} = \frac{1}{4\sin^2(\theta/2)}.
\end{equation*}
The sums corresponding to higher powers in the denominator can then be obtained by differentiating the above identity recursively with respect to $\theta$.
\eremk
\end{remark}
\begin{example} \label{ex:3.18}
With \eqref{eq:3.21} and the previous remark, we can explicitly compute
\begin{equation*}
\begin{aligned}
    B_2(\theta) & = -(2-2\cos \theta)^3 \sum_{j \in \Z} \frac{1}{(\theta + 2j\pi)^4} = -(2-2\cos \theta)^3 \frac{3-2\sin^2(\theta/2)}{48 \sin^4(\theta/2)} = - \frac{2}{3} (1-\cos \theta)(2+\cos \theta),
    \\
    C_2(\theta) & = -(2-2\cos \theta)^3 \sum_{j \in \Z} \frac{1}{(\theta + 2j\pi)^5} =  -(2-2\cos \theta)^3 \cos(\theta/2) \frac{3-\sin^2(\theta/2)}{96 \sin^5(\theta/2)}= - \frac{1}{6} \sin \theta (5+\cos \theta).
\end{aligned}
\end{equation*}
Substituting these expressions into \eqref{eq:3.22} and using the identities $\cos \theta = (z+z^{-1})/2$ and $\sin \theta = -\mi(z-z^{-1})/2$, for $z = e^{\mi \theta}$, we recover exactly the polynomials in~\eqref{eq:3.20}.
\end{example}
Combining Proposition \ref{prop:311} and \eqref{eq:3.19}, we obtain the following explicit expression for the symbol polynomial of the family $\{\K_n^p(\rho)\}_n$:
\begin{equation*}
     q^{\mathbf{K}^p(\rho)}(e^{\mi \theta}) =  -\mi e^{\mi p \theta}( B_p(\theta) - \rho C_p(\theta)), \quad \rho \in \R, \ \theta \in [-\pi,\pi], \text{~and~} p \in \N.
\end{equation*}
In view of Proposition \ref{prop:3.14}, we need to verify that, for all~$\rho \in \R$, $q^{\K^p(\rho)}(e^{\mi \theta})$ has exactly two simple zeros (as a function of~$\theta$). To this aim, we define the auxiliary function
\begin{equation*}
    K_p : [-\pi,\pi] \times \R \to \R^+, \quad K_p(\theta,\rho) := B_p(\theta) - \rho C_p(\theta).
\end{equation*}
Observe that $K_p(\theta,\rho) = 0$ if and only if $q^{K^p(\rho)}(e^{\mi \theta})=0$.
\begin{proposition}[Simple zeros of~$K_p(\theta, \rho)$] \label{prop:3.19}
For all $\rho \in \R$ and~$p \in \mathbb{N}$, the function $K_p(\theta,\rho)$ has exactly 2 simple zeros for $\theta \in [-\pi,\pi]$.
\end{proposition}
\begin{proof}
We begin by recalling some properties of the functions $ B_p$ and $C_p$. The corresponding proofs can be found in~\cite[Lemma 3.8]{FerrariFraschiniLoliPerugia2025} and~\cite[Cor.~5.4]{FerrariFraschini2026}:
\begin{equation} \label{eq:3.23}
\begin{aligned} 
     & \lim_{\theta \to 0} B_p(\theta) = \lim_{\theta \to 0}  C_p(\theta)  = 0 , \quad  C_p(-\pi)=C_p(\pi)=0, \quad B_p(\pi)=B_p(-\pi)<0,
    \\  & \lim_{\theta \to 0} B_p'(\theta) = 0, \quad \lim_{\theta \to 0} C_p'(\theta) = -1, \quad C_p(\theta)<0 \text{~for~} \theta \in (0,\pi), \quad B_p \text{~is even}, \quad C_p \text{~is odd}.
\end{aligned}
\end{equation}
We deduce that $\theta=0$ is a simple zero of~$K_p(\theta,\rho)$, as
\begin{equation*}
    \lim_{\theta \to 0} K_p(\theta,\rho) = 0, \quad \lim_{\theta \to 0} \partial_\theta K_p(\theta,\rho) = \rho.
\end{equation*}
It remains to show that, for all $\rho \in \R$, $K_p(\theta,\rho)=0$ has exactly one solution in $\theta \in [-\pi,0) \cup (0,\pi]$, and that such a zero is simple. For~$\theta \in (-\pi,0) \cup (0,\pi)$, we define the auxiliary function
\begin{equation} \label{eq:3.24}
    L_p(\theta,\rho) := \frac{K_p(\theta,\rho)}{C_p(\theta)} = \frac{B_p(\theta)}{C_p(\theta)} - \rho.
\end{equation}
The following limits are a direct consequence of~\eqref{eq:3.23}:
\begin{equation*}
    \lim_{\theta \to 0} L_p(\theta, \rho) = -\rho, \qquad \lim_{\theta \to -\pi^+} L_p(\theta, \rho) = -\infty, \qquad \lim_{\theta \to \pi^-} L_p(\theta, \rho) = +\infty.
\end{equation*}
Hence, the function $L_p(\theta, \rho)$ admits a continuous extension at $\theta = 0$. 

We now aim at proving that~$L_p(\theta, \rho)$ is strictly increasing on~$(-\pi, \pi)$. From this, we immediately deduce that the equation $L_p(\theta, \rho) = 0$ admits exactly one solution in $(-\pi, \pi)$, and that this zero is simple. Consequently, the same conclusion holds for $K_p(\theta, \rho)$. We now proceed to show the monotonicity of $L_p(\theta, \rho)$. Note that  $\partial_\theta L_p(\theta,\rho) = \partial_\theta (L_p(\theta,\rho)+\rho)$, and~$L_p(\theta,\rho)+\rho = B_p(\theta)/C_p(\theta)$ is an odd function in~$\theta$ that is independent of~$\rho$. Therefore, it is enough to study~$L_p(\theta,\rho)$ only for~$\theta \in (0,\pi)$.

For $k \in \mathbb{N}$ and $k \ge 2$, define the auxiliary functions $\widehat{U}_k: (0,\pi) \to \R$ as 
\begin{equation*}
    \widehat{U}_k(\theta) := \sum_{j \in \Z} \frac{1}{(\theta+2j\pi)^k},
\end{equation*}
and note that, using \eqref{eq:3.21}, for $\theta \in (0,\pi)$, it holds 
\begin{align*}
    \partial_\theta L_p(\theta,\rho) & = \partial_\theta  \left(\frac{\widehat{U}_{2p}(\theta)}{\widehat{U}_{2p+1}(\theta)}\right) = \frac{\widehat{U}'_{2p}(\theta)\widehat{U}_{2p+1}(\theta) - \widehat{U}_{2p}(\theta) \widehat{U}_{2p+1}'(\theta)}{(\widehat{U}_{2p+1}(\theta))^2}.
\end{align*}
Using the identity $\widehat{U}_k'(\theta) = -k \widehat{U}_{k+1}(\theta)$, we further obtain 
\begin{align*}
    \partial_\theta L_p(\theta,\rho) =  \frac{-2p(\widehat{U}_{2p+1}(\theta))^2 + (2p+1)\widehat{U}_{2p}(\theta) \widehat{U}_{2p+2}(\theta)}{(\widehat{U}_{2p+1}(\theta))^2}.
\end{align*}
It remains to show that, for all $\theta \in (0,\pi)$,
\begin{equation} \label{eq:3.25}
    2p\bigl(\widehat{U}_{2p+1}(\theta)\bigr)^2 - (2p+1) \widehat{U}_{2p}(\theta) \widehat{U}_{2p+2}(\theta) < 0.   
\end{equation}
It is shown in \cite[Lemma A.1]{FerrariFraschiniLoliPerugia2025} that $\widehat{U}_{2p+1}(\theta) < \theta \widehat{U}_{2p+2}(\theta)$ \footnote{In the notation of \cite{FerrariFraschiniLoliPerugia2025}, $\widehat{C}_p(\theta) = -\widehat{U}_{2p+1}(\theta)$ and $\widehat{B}_p(\theta) = -\widehat{U}_{2p}(\theta)$} and since $\widehat{U}_{2p+2}(\theta)>0$, then \eqref{eq:3.25} holds true provided that
\begin{equation*}
    2p (\theta^2 \widehat{U}_{2p+2}(\theta) - \widehat{U}_{2p}(\theta)) -  \widehat{U}_{2p}(\theta) < 0.   
\end{equation*}
This latter follows from $\widehat{U}_{2p}(\theta)>0$ and~\cite[Lemma 1]{EkstromFurciGaroniManniSerraCapizzanoSpeleers2018}, where it is shown $\theta^2 \widehat{U}_{2p+2}(\theta) < \widehat{U}_{2p}(\theta)$.\footnote{In the notation of \cite{EkstromFurciGaroniManniSerraCapizzanoSpeleers2018}, $e_p(\theta) = \widehat{U}_{2p}(\theta)/\widehat{U}_{2p+2}(\theta)$}
\end{proof}
\begin{example}
Using the expressions for $B_2(\theta)$ and $C_2(\theta)$ derived in Example \ref{ex:3.18}, we obtain
\begin{equation*}
    L_2(\theta,\rho) = 4 \Big(\frac{2-\cos\theta-\cos^2 \theta}{\sin \theta (5+\cos \theta)}\Big) - \rho, \quad \partial_\theta L_2(\theta,\rho) = 4 \Big(\frac{7 - 3\cos^2\theta - 4\cos^3\theta}{\sin^2\theta(5+\cos\theta)^2}\Big).
\end{equation*}
The polynomial $g(x) = 7 - 3x^2 -4x^3$ is strictly positive for all $x \in (-1, 1)$, thereby confirming that $L_2(\theta,\rho)$ is strictly increasing on $(-\pi,\pi)$ for all $\rho \in \R$. To visually show the presence of exactly two zeros, in Figure \ref{fig:2}, we display $K_2(\theta,\rho)$ as a function of $\theta$ for various values of $\rho$.

\begin{figure}[ht]
        \centering
        \includegraphics[width=0.6\linewidth]{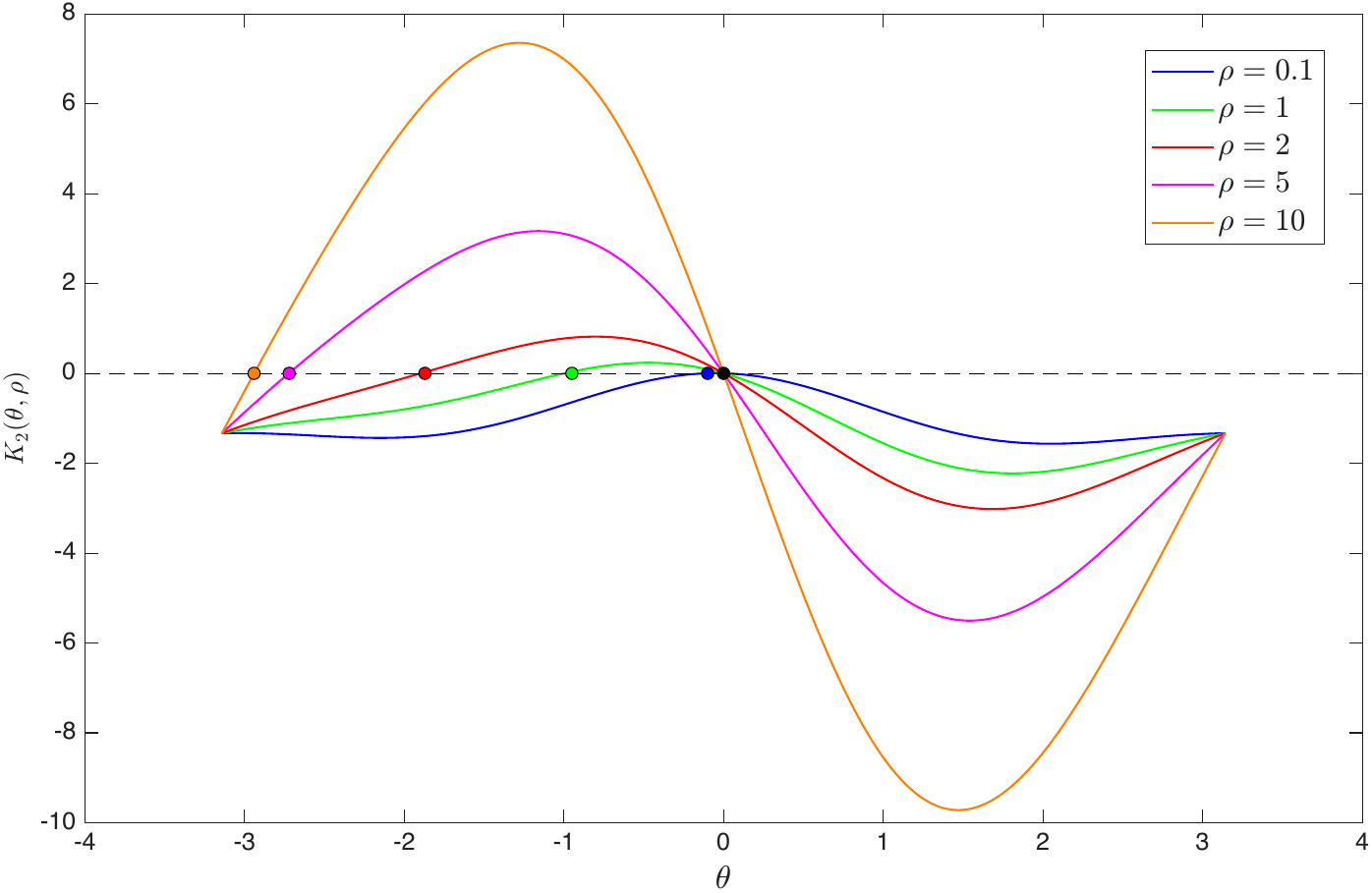}
        \caption{The function $K_2(\theta, \rho)$ for $\rho \in \{0.1, 1, 5, 10\}$. Each curve exhibits exactly two simple roots in~$[-\pi, \pi]$, one of which is at $\theta = 0$.}
        \label{fig:2}
\end{figure}
\end{example}
Combining Propositions \ref{prop:3.14} and \ref{prop:3.19}, we obtain the main result of the paper.
\begin{theorem}[Unconditional stability of~\eqref{eq:3.2}] \label{th:3.21}
Let Assumption \ref{assu:3.5} be satisfied. Then, for all $\rho \in \R$, the family of matrices~$\{\K_n^p(\rho)\}_n$ defined in~\eqref{eq:3.14} is weakly well-conditioned. Consequently, the variational formulation in \eqref{eq:3.2} is
unconditionally stable.
\end{theorem}
\begin{remark}[The Schr\"odinger and wave equations]
The discrete formulation~\eqref{eq:3.2} is unconditionally stable. This is in contrast to the space--time formulation for the wave equation analyzed in~\cite{FerrariFraschini2026}, which is only conditionally stable for all polynomial degrees. While a stabilization approach was proposed in~\cite{FraschiniLoliMoiolaSangalli2024} (and further analyzed in~\cite{FerrariFraschini2026}), an alternative formulation that achieves unconditional stability without requiring stabilization was introduced in~\cite{FerrariFraschiniLoliPerugia2025}. In the following section, we establish a strict connection between~\eqref{eq:3.2} and the method presented in~\cite{FerrariFraschiniLoliPerugia2025}.
\eremk
\end{remark}

In Figure~\ref{fig:3}, we present the condition numbers of the system matrices~$\{\K_n^p(\rho)\}_n$ 
with fixed size~$n=1000$, for~$p \in \{1, 2, 3, 4, 5, 6\}$, as the parameter~$\rho$ varies over the range~$[1, 100]$. As predicted by Theorem~\ref{th:3.21}, the condition numbers exhibit only mild dependence on~$\rho = \mu h_t$ and no CFL-type restrictions are observed or required to ensure stability.

\begin{figure}[ht]
        \centering
        \includegraphics[width=0.48\linewidth]{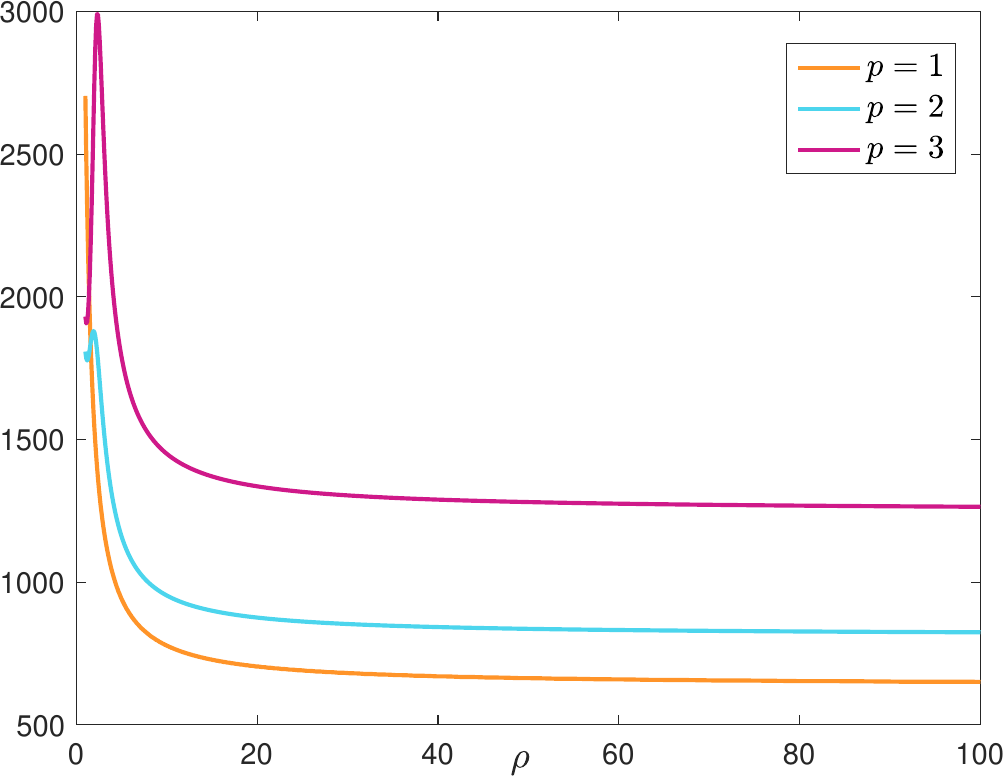}
        \includegraphics[width=0.495\linewidth]{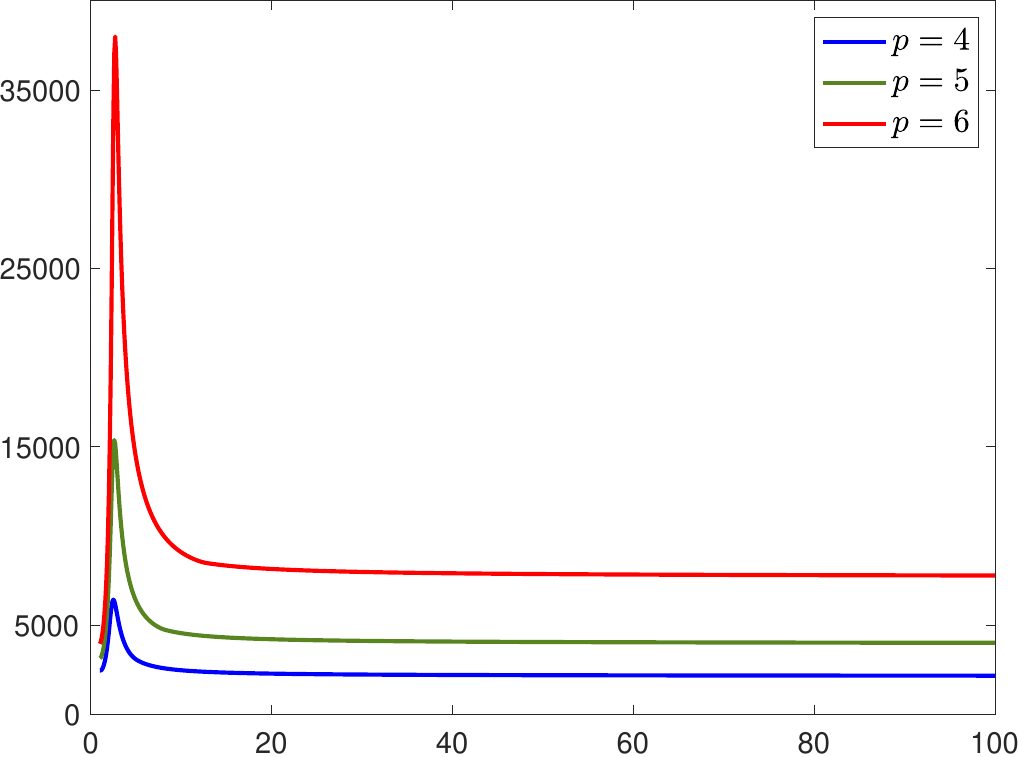}
        \caption{Spectral condition numbers of the system matrices~$\bigl\{\K_n^p(\rho)\bigr\}_n$ with fixed dimension~$n = 1000$, plotted as a function of the parameter~$\rho \in [1, 100]$. The results are shown for~$p \in \{1,2,3\}$ (left panel) and~$p \in \{4,5,6\}$ (right panel).}
        \label{fig:3}
\end{figure}

\section{First-order-in-time isogeometric method for the wave equation} \label{sec:4}

There is a natural algebraic connection between the variational formulation~\eqref{eq:3.2} and the first-order-in-time formulation  in~\cite[Eq. (6)]{FerrariFraschiniLoliPerugia2025} for the wave equation. In this section, we recall the ODE counterpart of the latter formulation, and exploit its connection with~\eqref{eq:3.2} to derive an alternative proof of its unconditional stability, which follows directly from Theorem~\ref{th:3.21}. We remark that the original proof of stability in~\cite{FerrariFraschiniLoliPerugia2025} is more involved and relies on a numerical verification, which has only been computed for spline spaces of degree up to~$p = 17$. Such a limitation is overcome in Theorem~\ref{th:4.4}, thus showing  that the method in~\cite{FerrariFraschiniLoliPerugia2025} is unconditionally stable for all~$p \in \N$.

Let $\mu_s \in \R$ denote the parameter associated with the IVP~\eqref{eq:3.1}, and let~$\mu_w > 0$ denote the corresponding parameter for the wave equation in its first-order-in-time formulation. Given $f \in L^2(0,T)$, we consider the following IVP: find $u, v \in H^1(0,T)$ such that

\begin{equation} \label{eq:4.1}
    \begin{cases}
    \dpt u(t) - v(t) = 0, & t \in (0,T), \\
    \dpt v(t) + \mu_w u(t) = f(t), & t \in (0,T), \\
    u(0) = 0, \quad v(0) = 0.
    \end{cases}
\end{equation}
\begin{remark}
Just as the initial value problem \eqref{eq:3.1} represents the ODE counterpart to the Schr\"odinger IBVP \eqref{eq:1.1}, system~\eqref{eq:4.1} corresponds to the following first-order-in-time formulation of the wave equation: given a source term $F : Q_T \to \R$, find the wave field $U : \overline{Q}_T \to \R$ and the velocity $V : \overline{Q}_T \to \R$ such that
\begin{equation*}
    \begin{cases}
        \partial_t U (\bx,t) - V(\bx,t) = 0, & (\bx,t) \in Q_T,
        \\ \partial_t V (\bx,t) - \Delta_{\bx} U(\bx,t) = F(\bx,t), & (\bx,t) \in Q_T,
        \\ U (\bx,t) = V(\bx,t) = 0 & (\bx,t) \in \partial \Omega \times J_T,
        \\ U(\bx,0) = V(\bx,0) = 0, & \bx \in \Omega.
    \end{cases}
\end{equation*}
\eremk
\end{remark}
The discrete variational formulation %
of~\eqref{eq:4.1} proposed in~\cite{FerrariFraschiniLoliPerugia2025} reads:
\begin{tcolorbox}[
    colframe=black!50!white,
    colback=blue!5!white,
    boxrule=0.5mm,
    sharp corners,
    boxsep=0.5mm,
    top=0.5mm,
    bottom=0.5mm,
    right=0.25mm,
    left=0.1mm
]
    \begingroup
    \setlength{\abovedisplayskip}{0pt}
    \setlength{\belowdisplayskip}{0pt}
    \, find~$u_{h_t}^p \in S_{h_t,0,\bullet}^p(0,T;\R)$ and $v_{h_t}^p \in S_{h_t,0,\bullet}^p(0,T;\R)$ such that \vspace{0.1cm}
\begin{equation} \label{eq:4.2}
	\begin{cases}
		(\dpt u^p_{h_t}, \dpt \chi_{h_{t}}^p)_{J_T} + (\dpt v_{h_t}^p, \chi^p_{h_t})_{J_T} = 0 & \text{for all~} \chi_{h_t}^p \in S_{h_t,\bullet,0}^p(0,T;\R),
		\\ (\partial_t v_{h_t}^p, \partial_t \lambda_{h_t}^p)_{J_T} - \mu_w (\partial_t u_{h_t}^p, \lambda_{h_{t}}^p)_{J_T} = (f,\partial_t \lambda_{h_{t}}^p)_{J_T} & \text{for all~} \lambda_{h_{t}}^p \in S_{h_{t},\bullet,0}^p(0,T;\R).
	\end{cases}
\end{equation}
    \endgroup
\end{tcolorbox}
In the next result, we highlight the connection between the variational formulations~\eqref{eq:3.2} and~\eqref{eq:4.2}.
\begin{proposition}[Relation between the discrete formulations~\eqref{eq:3.2} and~\eqref{eq:4.2}] \label{prop:41}
Let $f = f_r + \mi f_i$, with real-valued functions~$f_r$ and~$f_i$. Then,
\begin{equation*}
    \psi_{h_{t}}^p = u_{h_{t}}^p~+ \mi v_{h_{t}}^p \in~S_{h_{t}}^p(0,T;\C)
\end{equation*}
is a solution to problem~\eqref{eq:3.2} with $\Psi_0=0$ if and only if the real-valued functions $u_{h_{t}}^p \in S^p_{h_{t},0,\bullet}(0,T;\R)$ and $v_{h_{t}}^p \in S_{h_{t},0,\bullet}^p(0,T;\R)$ satisfy
\begin{equation} \label{eq:4.3}
\begin{cases}
    (\dpt u_{h_{t}}^p, \chi_{h_{t}}^{p-1})_{J_T} + \mu_s (v_{h_{t}}^p, \chi_{h_{t}}^{p-1})_{J_T} = (f_i, \chi_{h_{t}}^{p-1})_{J_T} & \text{for all}~\chi_{h_{t}}^{p-1} \in S_{h_{t}}^{p-1}(0,T;\R),
    \\ (\dpt v_{h_{t}}^p, \lambda_{h_{t}}^{p-1})_{J_T} - \mu_s (u_{h_{t}}^p, \lambda_{h_t}^{p-1})_{J_T} = - (f_r, \lambda_{h_{t}}^{p-1})_{J_T} & \text{for all}~\lambda_{h_{t}}^{p-1} \in S_{h_{t}}^{p-1}(0,T;\R),
\end{cases}
\end{equation}
or, equivalently,
\begin{equation} \label{eq:4.4}
\begin{cases}
    (\dpt u_{h_{t}}^p, \dpt \chi_{h_{t}}^p)_{J_T} - \mu_s (\dpt v_{h_{t}}^p, \chi_{h_{t}}^p)_{J_T} = (f_i, \dpt \chi_{h_{t}}^p)_{J_T} & \text{for all}~\chi_{h_{t}}^p \in S_{h_{t},\bullet,0}^p(0,T;\R),
    \\ (\dpt v_{h_{t}}^p, \dpt \lambda_{h_{t}}^p)_{J_T} + \mu_s (\dpt u_{h_{t}}^p, \lambda_{h_t}^p)_{J_T} = - (f_r, \dpt \lambda_{h_t}^p)_{J_T} & \text{for all}~\lambda_{h_t}^p \in S_{h_t,\bullet,0}^p(0,T;\R).
\end{cases}
\end{equation}
\end{proposition}
\begin{proof}
Setting $\psi_{h_t}^p = u_{h_t}^p + \mi v_{h_t}^p$, then problem \eqref{eq:3.2} is equivalent to: find $u_{h_t}^p \in S^p_{h_t,0,\bullet}(0,T;\R)$ and $v_{h_t}^p \in  S_{h_t,0,\bullet}^p(0,T;\R)$ such that
\begin{equation*}
    \mi (\dpt u_{h_t}^p, w_{h_t}^{p-1})_{J_T} - (\dpt v_{h_t}^p, w_{h_t}^{p-1})_{J_T} + \mu_s (u_{h_t}^p,  w_{h_t}^{p-1})_{J_T} + \mi \mu_s (v_{h_t}^p, w_{h_t}^{p-1})_{J_T} = (f_r, w_{h_t}^{p-1})_{J_T} + \mi (f_i, w_{h_t}^{p-1})_{J_T}
\end{equation*}
for all $w_{h_t}^{p-1} \in S_{h_t}^{p-1}(0,T;\R)$.  Finally, \eqref{eq:4.3} follows by taking the imaginary and real parts of this equation. 
\end{proof}
With the data $\mu_s=-1$, $\mu_w=1$, $f_r = -f$, and~$f_i=0$ problem \eqref{eq:4.4} coincides with the variational formulation \eqref{eq:4.2}. Therefore, we expect the two methods to exhibit similar stability properties. In fact, both methods turn out to be unconditionally stable. Here, we show that the stability of method~\eqref{eq:4.2} is a consequence of Theorem \ref{th:3.21}. Before, we state an auxiliary result.
\begin{lemma} \label{lem:4.3}
Consider an invertible $2 \times 2$ block matrix 
\begin{equation*}
\mathbf{M} =
    \begin{pmatrix}
        \B & -\bC \\
        \bC & \B
    \end{pmatrix}
\end{equation*}
with $\B, \bC \in \R^{n \times n}$, for some 
$n \in \mathbb{N}$. Suppose also that~$\B$ is invertible, and define the Schur complement
\begin{equation*} 
    \bS :=\B+\bC\mathbf{B}^{-1} \bC.
\end{equation*}
Then, the condition number of~$\bS$ satisfies the following bound:
\begin{equation*}
    \kappa (\bS)  \le \left(1 + \frac{\| \bC \|^2}{\| \B\|^2} \kappa(\B)\right) \kappa(\bM),
\end{equation*}
where $\| \cdot \|$ denotes either the $\| \cdot \|_\infty$ or the $\| \cdot \|_1$ norm, and the condition number is computed with respect to the same norm.
\end{lemma}

\begin{proof}
First, we compute 
\begin{equation} \label{eq:4.5}
    \| \bS \| \le \| \B \| + \| \B^{-1} \| \| \bC \|^2.
\end{equation}
Then, using the inversion formula from  \cite[\S 0.8.5]{HornJohnson2013}, we get 
\begin{equation*}
\bM^{-1} = 
\begin{pmatrix}
    \B^{-1} - \B^{-1} \bC \bS^{-1} \bC \B^{-1} & \B^{-1} \bC \bS^{-1} \\
    -\bS^{-1} \bC \B & \bS^{-1}
\end{pmatrix}.
\end{equation*}
Being a subpart of the full matrix, we deduce $\| \bS^{-1} \| \le \| \bM^{-1} \|$. Similarly, we also know that~$\| \B \| \le \| \bM\|$. Combining these two estimates with~\eqref{eq:4.5}, we conclude
\begin{align*}
    \kappa (\bS) = \| \bS \| \| \bS^{-1} \| \le (\| \B \| + \| \B^{-1} \| \| \bC \|^2) \| \bM^{-1} \| & \le \left(\| \B \| + \| \bC \|^2 \frac{\kappa(\B)}{\| \B \|}\right) \frac{\kappa(\bM)}{\| \bM \|} 
    \\ & \le \left(1 + \frac{\| \bC \|^2}{\| \B\|^2} \kappa(\B)\right) \kappa(\bM),
\end{align*}
which completes the proof.
\end{proof}
\begin{theorem}[Unconditional stability of~\eqref{eq:4.2}] \label{th:4.4}
Let Assumption~\ref{assu:3.5} be satisfied. Then, the discrete formulation \eqref{eq:4.2} is unconditionally stable.
\end{theorem}
\begin{proof}
Recalling the definitions of $\B_{h_t}^p$ and $\bC_{h_t}^p$ from~\eqref{eq:3.10}, we observe that the system matrices associated with formulations \eqref{eq:4.4} and \eqref{eq:4.2} are, respectively,
\begin{equation} \label{eq:4.6}
\begin{pmatrix}
    \B_{h_t}^p & -\mu_s \bC_{h_t}^p \\
    \mu_s \bC_{h_t}^p & \B_{h_t}^p 
\end{pmatrix} 
\quad\quad\text{and}\quad\quad 
\begin{pmatrix}
    \B_{h_t}^p &  \bC_{h_t}^p \\
    -\mu_w \bC_{h_t}^p & \B_{h_t}^p 
\end{pmatrix}.
\end{equation}
In~\cite[Prop.~3.10]{FerrariFraschiniLoliPerugia2025}, it was shown that the matrices~$\{\B_{h_t}^p\}_{h_t}$ are invertible. Therefore, the Schur complements
\begin{equation*} 
     \B_{h_t}^p + \mu_s^2 \bC_{h_t}^p (\B_{h_t}^p)^{-1} \bC_{h_t}^p \quad \text{and} \quad \B_{h_t}^p + \mu_w \bC_{h_t}^p (\B_{h_t}^p)^{-1} \bC_{h_t}^p,
\end{equation*}
associated with the matrices in \eqref{eq:4.6} are both well defined. Theorem \ref{th:3.21} guarantees the stability of the variational formulation~\eqref{eq:3.2}, and, consequently, that of~\eqref{eq:4.4}. Moreover, Lemma \ref{lem:4.3} implies that the corresponding family of Schur complements is weakly well-conditioned for all $\mu_s \in \R$. This follows from the fact that $\| \bC_{h_t}^p\|$ depends only on $p$ (see property~\ref{P1} in Proposition~\ref{prop:3.8}), and that the family~$\{\B_{h_t}^p\}_{h_t}$ is weakly well-conditioned (see~\cite[Prop.~3.10]{FerrariFraschiniLoliPerugia2025}). By choosing~$\mu_s = \sqrt{\mu_w}$, the Schur complements for both matrices in~\eqref{eq:4.6} coincide. Hence, the weakly well-conditioning of these Schur complements implies the stability of formulation~\eqref{eq:4.2}.
\end{proof}
\begin{remark}[Improvement of the stability result in~\cite{FerrariFraschiniLoliPerugia2025}]
In \cite{FerrariFraschiniLoliPerugia2025}, the unconditional stability of the discrete variational formulation~\eqref{eq:4.2} was studied by analyzing the condition number of the associated Schur complements. However, checking the hypotheses of Theorem \ref{th:3.11} for the perturbed blocks  relied on numerical computations, which restricted the stability result to splines of degree~$p \le 17$ (see \cite[\S3.3]{FerrariFraschiniLoliPerugia2025}). The algebraic connection with the Schr\"odinger formulation~\eqref{eq:3.2} allows us to bypass this constraint. Consequently, Theorem~\ref{th:4.4} now provides a rigorous proof of unconditional stability for any $p \in \N$, without relying on numerical verifications.
\eremk
\end{remark}

\section{Numerical results} \label{sec:5}
In this section, we present an efficient implementation strategy for the discretization of the space--time method~\eqref{eq:2.6} (Section~\ref{sec:51}). Subsequently, using an isogeometric discretization also in space, 
we validate the theoretical stability results established in this work and demonstrate the quasi-optimal $h$-convergence and conservation properties of the proposed method (Section~\ref{sec:5.2}).\footnote{All numerical test are performed with MATLAB R2025b. The codes used for the numerical tests are available in the GitHub repository \cite{FerrariCodes2025}.}

\subsection{Solver} \label{sec:51}
To express the matrix form of the discrete variational formulation of~\eqref{eq:3.2}, we  denote by~$\bC^p_{h_t}$ and~$\B^p_{h_t}$ the temporal matrices introduced in~\eqref{eq:3.10}. Likewise, let~$\bM_{h_{\bx}}$ and~$\A_{h_{\bx}}$ denote the spatial mass matrix and the matrix representing the discretization of the Hamiltonian operator, respectively. The resulting global linear system takes the form
\begin{equation} \label{eq:5.1}
    (\mi \B_{h_t}^p \otimes \M_{h_{\bx}} + \bC_{h_t}^p \otimes \A_{h_{\bx}}) \boldsymbol{\Psi}^p_{\bh} = \mathbf{F}_{\bh}^p
\end{equation}
where~$\otimes$ denotes the Kronecker product. The vector $\boldsymbol{\Psi}_{\bh}^p$ collects the coefficients of the discrete unknown $\Psi_{\bh}^p$ with respect to the chosen basis of
$S_{h_t,0,\bullet}^p(0,T) \otimes S_{h_{\bx}}(\Omega)$ ordered appropriately. Similarly, $\mathbf{F}^p_{\bh}$ represents the vector form of the term on the right-hand side of~\eqref{eq:2.6}.

We now present an efficient algorithm for solving the linear system \eqref{eq:5.1}. This procedure is essentially a variant of the \emph{Bartels–Stewart method} \cite{BartelsStewart1972}, which relies on the complex Schur decomposition of the matrix $\mi (\B_{h_t}^p)^{-1} \bC_{h_t}^p$. Specifically, the involved decomposition reads
\begin{equation*}
    (\mathbf{Q}_{h_t}^p)^H (\mi \B_{h_t}^p)^{-1} \bC_{h_t}^p \mathbf{Q}_{h_t}^p = \mathbf{R}_{h_t}^p,
\end{equation*}
where~$\mathbf{Q}_{h_t}^p$ is a unitary matrix, the superscript~$H$ denotes the conjugate transpose, and~$\mathbf{R}_{h_t}^p$ is an upper triangular matrix.

The main steps of the algorithm are summarized in Algorithm~\ref{alg:1} (see also \cite{FerrariFraschiniLoliPerugia2025, LangerZank2021, Tani2017} for similar algorithms).

\begin{algorithm}[ht!]
\caption{Efficient implementation for solving the system \eqref{eq:5.1}} \label{alg:1}
\begin{algorithmic}[1]
    \State{\textbf{Compute} the complex Schur decomposition $(\mathbf{Q}_{h_t}^p,\mathbf{R}_{h_t}^p)$ of $(\mi \B_{h_t}^p)^{-1} \bC_{h_t}^p$}
    \vspace{0.1cm}
    \State{\textbf{Solve} for $\boldsymbol{Y}_{\hspace{-0.09cm}\bh}^p$ the system $\left({\B}^{p}_{h_t} \otimes \mathbf{I}_{h_{\bx}} \right) \boldsymbol{Y}_{\hspace{-0.09cm} \bh}^p = \mathbf{F}_{\bh}^p$}
    \vspace{0.1cm}
    \State{\textbf{Update} $\boldsymbol{Y}_{\hspace{-0.09cm} \bh}^p \leftarrow \left((\mathbf{Q}_{h_t}^p)^H \otimes \mathbf{I}_{h_{\bx}}\right) \boldsymbol{Y}_{\hspace{-0.09cm} \bh}^p$}
    \vspace{0.1cm}
    \State{\textbf{Solve} for $\boldsymbol{Z}_{\hspace{-0.04cm} \bh}^p$ the system $\left(\mathbf{I}_{h_t} \otimes \M_{h_{\bx}} + \mathbf{R}^p_{h_t} \otimes \A_{h_{\bx}} \right) \boldsymbol{Z}_{\bh}^p =\boldsymbol{Y}_{\bh}^p$}
    \vspace{0.1cm}
    \State{\textbf{Compute} the final solution~$\boldsymbol{\Psi}_{\bh}^p = 
    (\mathbf{Q}_{h_t}^p \otimes \mathbf{I}_{h_{\bx}}) \boldsymbol{Z}_{\bh}^p$}
\end{algorithmic}
\end{algorithm}

In particular, \textit{step} 2 only requires solving a set of completely independent small linear systems with the same matrix~$\B_{h_t}^p$, \textit{steps} 3 and 5 involve just matrix-vector products, and the matrix of the linear system in \textit{step} 4 has a block upper triangular structure. For a detailed analysis of the computational cost associated with this algorithm, we refer the reader to~\cite[\S 5.1]{FerrariFraschiniLoliPerugia2025}.

\subsection{Numerical experiments} \label{sec:5.2}
We now consider some numerical experiments on the $(1+1)$-dimensional cylinder $Q_T = \Omega \times (0,1)$ with $\Omega = (-3,3)$. The exact solution is chosen from the following well-known family (see, e.g., \cite[\S2.3]{Griffiths1995}):
\begin{equation} \label{eq:5.2}
    \Psi^{(n)}(x,t) = \frac{1}{\sqrt{2^n n!}} \left( \frac{\omega}{\pi} \right)^{1/4}
    H_n\left( \sqrt{\omega} x \right) \exp\left( -\frac{1}{2} (\omega x^2 + (2n+1)i \omega t) \right), \quad n \in \mathbb{N},
\end{equation}
where $H_n(\cdot)$ is the $n$th physicist's Hermite polynomial, as defined in \cite[Table 18.3.1]{OlverFrankLozeirClark2010}. These functions are exact solutions of the Schrödinger equation with potential $V(x) = -\omega^2 x^2/2$, modeling a quantum harmonic oscillator with angular frequency $\omega > 0$. The source term is $F(x,t) = 0$, and the initial condition is given by $\Psi_0(x) = \Psi^{(n)}(x,0)$. In our tests, we fix $\omega = 10$ and~$n = 2$.

For the space–time Galerkin discretization \eqref{eq:2.6}, we employ maximal regularity splines over uniform meshes for both the spatial and temporal domains. The discrete spaces~$S_{h_x}(\Omega)$ and~$S_{h_t}^{p_t}(0,T)$ use the same polynomial degree~$p_x = p_t = p$, but we allow for different meshsizes~$h_x$ and $h_t$. The initial condition is approximated by $\Pi_x \Psi_0$, which denotes the~$L^2(\Omega)$ projection of~$\Psi_0$ onto~$S_{h_x}(\Omega)$.

We compute the relative errors in the following norms
\begin{equation} \label{eq:5.3}
\begin{aligned}
     \| \Psi \|_{H^1(Q_T)} := \| \partial_t \Psi \|_{L^2(Q_T)} + \| \partial_x \Psi \|_{L^2(Q_T)} \quad\quad \| \Psi \|_{L^2(Q_T)} := \left(\int_0^T \int_\Omega |\Psi(x,t)|^2 \, \dd x \, \dd t\right)^{\frac{1}{2}}.
    \end{aligned}
\end{equation}

\paragraph{Stability.}
To verify the unconditional stability of the method (i.e., absence of a CFL-type restriction such as %
$h_t \le C h_x^2$), we fix~$h_t = 0.0625$ and progressively decrease~$h_x$. Figure~\ref{fig:4} shows the relative $L^2(Q_T)$ and $H^1(Q_T)$ errors for different polynomial degrees. The results confirm that the method remains stable and free of spurious oscillations even for highly refined spatial meshes.

\begin{figure}[ht]
\centering
\begin{minipage}{0.475\textwidth}
\begin{tikzpicture}
\begin{groupplot}[group style={group size=1 by 1},height=7cm,width=8cm, every axis label={font=\normalsize}, ylabel style={font=\footnotesize}]
\nextgroupplot[xtick={1,10,100},
                xmode=log,
                ytick={0.0316,0.1,0.00316227766,0.01,0.001},
                xlabel={$h_t/h_x$},
                ymode=log,
                title=$H^1$-error,
                legend pos=south west, 
                legend style={nodes={scale=1, transform shape}},]
    \pgfplotstableread{
        h  error
        1     0.047478679385498
        2     0.046977801834104
        4     0.046950331310437
        8     0.046947589949453
        16    0.046947412084121
        32    0.046947389678381
        64    0.046947388387963
        128   0.046947388423246
        } \datatable
        \addplot[mark=pentagon*, magenta,mark size=2.5] table[x=h, y=error] \datatable;
        \pgfplotstableread{
        h  error
        1     0.010862578144316
        2     0.010861779396991
        4     0.010861795437006
        8     0.010861793386087
        16    0.010861793384487
        32    0.010861793383401
        64    0.010861793373255
        128   0.010861793346100
        } \datatable
        \addplot[mark=square*, blue] table[x=h, y=error] \datatable;
        \pgfplotstableread{
        h  error
        1     0.003403944482192
        2     0.003403926295218
        4     0.003403926692972
        8     0.003403926674471
        16    0.003403926674410
        32    0.003403926674194
        64    0.003403926674115
        128   0.003403926675144
        } \datatable
        \addplot[mark=diamond*, darkgreen, mark size=3] table[x=h, y=error] \datatable;
    \end{groupplot}
    \end{tikzpicture}
\end{minipage}
\begin{minipage}{0.475\textwidth}
\begin{tikzpicture}
\begin{groupplot}[group style={group size=1 by 1},height=7cm,width=8cm, every axis label={font=\normalsize}, ylabel style={font=\footnotesize}]
\nextgroupplot[xtick={1,10,100},
                xmode=log,
                ytick={0.1,0.01,0.00316227766,0.001},
                xlabel={$h_t/h_{\bx}$},
                ymode=log,
                title=$L^2$-error,
                legend pos=south west, 
                legend style={nodes={scale=1, transform shape}},]
        \pgfplotstableread{
        h  error
        1     0.028198785291818              
        2     0.028197738002712
        4     0.028197742726550
        8     0.028197751879357
        16    0.028197751884904
        32    0.028197751885363
        64    0.028197751908872 
        128   0.028197751937356
        } \datatable
        \addplot[mark=pentagon*, magenta,mark size=2.5] table[x=h, y=error] \datatable;
        \pgfplotstableread{
        h  error
        1     0.004813540302450               
        2     0.004813441142813
        4     0.004813441543200
        8     0.004813441674147
        16    0.004813441673795
        32    0.004813441671699
        64    0.004813441653773
        128   0.004813441600993
        } \datatable
        \addplot[mark=square*, blue, mark size=2] table[x=h, y=error] \datatable;
        \pgfplotstableread{
        h  error
        1     0.001293352388967
        2     0.001293345993245
        4     0.001293346025270
        8     0.001293346026351
        16    0.001293346026253 
        32    0.001293346025800
        64    0.001293346025642
        128   0.001293346027654
        } \datatable
        \addplot[mark=diamond*, darkgreen, mark size=3] table[x=h, y=error] \datatable;
    \end{groupplot}
    \end{tikzpicture}
\end{minipage}
\caption{Relative errors in the norms defined in \eqref{eq:5.3}, with maximal regularity splines in both space and time for polynomial degrees~$p=3$ ({\footnotesize{\textcolor{magenta}{$\pentagofill$}}} marker), $p=4$ ({\footnotesize{\textcolor{blue}{\(\blacksquare\)}}} marker) and $p=5$ ({\footnotesize{\textcolor{darkgreen}{\(\blacklozenge\)}}} marker), and the exact solution as in \eqref{eq:5.2} with $\omega=10$ and $n=2$. Here,~$h_t = 0.0625$ and~$h_{\bx}$ decreases.}
\label{fig:4}
\end{figure}

\paragraph{Convergence.}
In Figure \ref{fig:5}, we present the relative errors obtained for various degrees of approximation, and a sequence of tensor-product meshes with~$h = h_{\bx} = h_t = 2^{-j}$, $j = 4, \ldots, 8$. We observe optimal convergence rates of order~$\mathcal{O}(h^p)$ in the $H^1(\QT)$ norm and~$\mathcal{O}(h^{p+1})$ in the~$L^2(\QT)$ norm, which is consistent with the theoretical approximation properties of maximal regularity B-splines of degree~$p$.

\begin{figure}[ht]
\centering
\begin{minipage}{0.475\textwidth}
\begin{tikzpicture}
\begin{groupplot}[group style={group size=2 by 1},height=7cm,width=8cm, every axis label={font=\normalsize}]
        
    \nextgroupplot[xmode=log, 
                log x ticks with fixed point,
                xtick={0.01,0.05},
                title = {$H^1$-error},
                ytick={10,0.1,0.001,0.00001,0.0000001,0.000000001},
                xlabel={$h_t=h_{\bx}$},
                ymode=log,
                legend pos=south east, 
                legend style={nodes={scale=1, transform shape}, font=\footnotesize},]
        \pgfplotstableread{
            x       y
            0.0625    1.357035020764602
            0.03125   0.638385654416084
            0.015625  0.200076354447084
            0.0078125 0.068247605954625
            0.00390625    0.027784645946226
        } \Lduepdue
    \addplot[orange, mark=*, mark options={scale=1}, line width=0.02cm] table[x=x, y=y] \Lduepdue;
        \pgfplotstableread{
            x       y
        0.0625    0.253857701252942
        0.03125   0.023835358579326
        0.015625  0.005068740777673
        0.0078125 0.001225739435411
        0.00390625    0.000302970081334
        } \Hunopdue     
        \addplot[cyan, mark=triangle*, mark options={scale=1.5}, line width=0.02cm] table[x=x, y=y] \Hunopdue;
        \pgfplotstableread{
            x       y
        0.0625    0.041845120620108
        0.03125   0.002753359392370
        0.015625  0.000304954380809
        0.0078125 0.000036781044597
        0.00390625    0.000004539663008
        } \Hduepdue
        \addplot[magenta, mark=pentagon*, mark options={scale=1.5}, line width=0.02cm] table[x=x, y=y] \Hduepdue;
        \pgfplotstableread{
            x       y
        0.0625    0.010862578144316
        0.03125   0.000365687608586
        0.015625  0.000018802145708
        0.0078125 0.000001114806952
        0.00390625    0.000000068042421
        } \Hduepdue
        \addplot[blue, mark=square*, mark options={scale=1.25}, line width=0.02cm] table[x=x, y=y] \Hduepdue;
        \pgfplotstableread{
            x       y
        0.0625    0.003403944482192
        0.03125   0.000050923065959
        0.015625  0.000001231036962 
        0.0078125 0.000000035474012
        0.00390625 0.000000001081432
        } \Hduepdue
        \addplot[darkgreen, mark=diamond*, mark options={scale=1.5}, line width=0.02cm] table[x=x, y=y] \Hduepdue;

        \logLogSlopeTriangle{0.26}{0.16}{0.725}{1}{orange};        
        \logLogSlopeTriangle{0.265}{0.16}{0.565}{2}{cyan};        
        \logLogSlopeTriangle{0.25}{0.13}{0.39}{3}{magenta};
        \logLogSlopeTriangle{0.25}{0.12}{0.25}{4}{blue};
        \logLogSlopeTriangle{0.25}{0.13}{0.09}{5}{darkgreen};
        
    \end{groupplot}
\end{tikzpicture}
\end{minipage}
\begin{minipage}{0.475\textwidth}
\begin{tikzpicture}
\begin{groupplot}[group style={group size=2 by 1},height=7cm,width=8cm, every axis label={font=\normalsize}]
        
    \nextgroupplot[xmode=log, 
                log x ticks with fixed point,
                xtick={0.01,0.05},
                title = {$L^2$-error},
                ytick={10,0.1,0.001,0.00001,0.0000001,0.000000001,0.00000000001},
                xlabel={$h_t=h_{\bx}$},
                ymode=log,
                legend pos=south east, 
                legend style={nodes={scale=1, transform shape}},]
        \pgfplotstableread{
            x       y
            0.0625     1.407040000863541
            0.03125    0.621971142145927
            0.015625   0.173435635185756
            0.0078125  0.044372238775190
            0.00390625 0.011154343161900
        } \Lduepdue
    \addplot[orange, mark=*, mark options={scale=1}, line width=0.02cm] table[x=x, y=y] \Lduepdue;
        \pgfplotstableread{
            x       y
        0.0625     0.230630481143431
        0.03125    0.009804855065121
        0.015625   0.000455993165612
        0.0078125  0.000059247738199
        0.00390625 0.000005667186152
        } \Hunopdue     
        \addplot[cyan, mark=triangle*, mark options={scale=1.5}, line width=0.02cm] table[x=x, y=y] \Hunopdue;
        \pgfplotstableread{
            x       y
        0.0625     0.028198785291818
        0.03125    0.000608602118236
        0.015625   0.000022335843227
        0.0078125  0.000001329167787
        0.00390625 0.000000081403195
        } \Hduepdue
        \addplot[magenta, mark=pentagon*, mark options={scale=1.5}, line width=0.02cm] table[x=x, y=y] \Hduepdue;
        \pgfplotstableread{
            x       y
        0.0625     0.004813540302450
        0.03125    0.000059247738199
        0.015625   0.000001455341447
        0.0078125  0.000000042127722
        0.00390625 0.000000001285164
        } \Hduepdue
        \addplot[blue, mark=square*, mark options={scale=1.1}, line width=0.02cm] table[x=x, y=y] \Hduepdue;
        \pgfplotstableread{
            x       y
        0.0625     0.001293352388967
        0.03125    0.000008230415161
        0.015625   0.000000092087549
        0.0078125  0.000000001306525
        0.00390625 0.000000000019711
        } \Hduepdue
        \addplot[darkgreen, mark=diamond*, mark options={scale=1.5}, line width=0.02cm] table[x=x, y=y] \Hduepdue;
                  
        \logLogSlopeTriangle{0.25}{0.2}{0.72}{2}{orange};        
        \logLogSlopeTriangle{0.27}{0.15}{0.49}{3}{cyan};        
        \logLogSlopeTriangle{0.25}{0.13}{0.36}{4}{magenta};
        \logLogSlopeTriangle{0.25}{0.12}{0.235}{5}{blue};
        \logLogSlopeTriangle{0.25}{0.13}{0.09}{6}{darkgreen};
        
    \end{groupplot}
\end{tikzpicture}
\end{minipage}
\caption{Relative errors in the~$H^1(\QT)$ norm (left plot) and the~$L^2(\QT)$ norm (right plot) for varying~$h_t = h_{\bx}$ with maximal regularity splines in space and time with polynomial degree~$p=1$ ({\large{\textcolor{orange}{$\bullet$}}}~marker), $p=2$ ({\footnotesize{\textcolor{cyan}{$\blacktriangle$}}} marker), $p=3$ ({\footnotesize{\textcolor{magenta}{$\pentagofill$}}} marker), $p=4$ ({\footnotesize{\textcolor{blue}{\(\blacksquare\)}}} marker) and $p=5$ ({\footnotesize{\textcolor{darkgreen}{\(\blacklozenge\)}}} marker). The exact solution is as in~\eqref{eq:5.2} with~$\omega=10$ and~$n=2$.}
\label{fig:5}
\end{figure}

\begin{figure}[ht]
\centering
\begin{minipage}{0.475\textwidth}
\begin{tikzpicture}
\begin{groupplot}[group style={group size=1 by 1},height=7cm,width=8cm, every axis label={font=\normalsize}, ylabel style={font=\footnotesize}]
        
    \nextgroupplot[xtick={0,0.2,0.4,0.6,0.8,1},
                ytick={0.1,0.001,0.00001,0.0000001,0.000000001,0.00000000001,0.00000000001,0.0000000000001,0.000000000000001},
                xlabel={$t$},
                ymode=log,
                title=$|\mathcal{M}_{\Psi_{\bh}}(t) - \mathcal{M}_{\Psi_{\bh}}(0)|$,
                legend pos=south west, 
                legend style={nodes={scale=1, transform shape}},]
    \pgfplotstableread{
        t       p2              p3              p4              p5           
                           0                   0                   0                   0                   0
   0.015625000000000   0.000414811616056   0.000058313777224   0.000000602750106   0.000000248919764
   0.031250000000000   0.000322552939036   0.000070150990864   0.000001866100030   0.000000265756105
   0.046875000000000   0.000336518724560   0.000065308370525   0.000000737367920   0.000000255051424
   0.062500000000000   0.000336497203784   0.000066240542326   0.000001486084800   0.000000247688008
   0.078125000000000   0.000335515846916   0.000066435045689   0.000001079723282   0.000000261114660
   0.093750000000000   0.000335979153897   0.000066112646742   0.000001259540676   0.000000249038810
   0.109375000000000   0.000335830815792   0.000066325407382   0.000001201630646   0.000000257352917
   0.125000000000000   0.000335867724006   0.000066220797917   0.000001206275618   0.000000252650538
   0.140625000000000   0.000335860888054   0.000066262224369   0.000001218330042   0.000000254769851
   0.156250000000000   0.000335861484425   0.000066249574329   0.000001205341043   0.000000254165805
   0.171875000000000   0.000335861689121   0.000066251687606   0.000001214462127   0.000000254050615
   0.187500000000000   0.000335861550464   0.000066252391656   0.000001209349995   0.000000254400772
   0.203125000000000   0.000335861601179   0.000066251458760   0.000001211689102   0.000000254056566
   0.218750000000000   0.000335861587162   0.000066252048371   0.000001210887059   0.000000254307488
   0.234375000000000   0.000335861590160   0.000066251765754   0.000001210994541   0.000000254157393
   0.250000000000000   0.000335861589747   0.000066251874946   0.000001211119752   0.000000254230557
   0.265625000000000   0.000335861589731   0.000066251842876   0.000001210967905   0.000000254205143
   0.281250000000000   0.000335861589769   0.000066251847464   0.000001211079021   0.000000254206207
   0.296875000000000   0.000335861589752   0.000066251849839   0.000001211014915   0.000000254214408
   0.312500000000000   0.000335861589756   0.000066251847157   0.000001211045179   0.000000254204956
   0.328125000000000   0.000335861589753   0.000066251848784   0.000001211034220   0.000000254212367
   0.343750000000000   0.000335861589753   0.000066251848020   0.000001211036173   0.000000254207670
   0.359375000000000   0.000335861589751   0.000066251848306   0.000001211037415   0.000000254210122
   0.375000000000000   0.000335861589750   0.000066251848225   0.000001211035656   0.000000254209148
   0.390625000000000   0.000335861589749   0.000066251848235   0.000001211037004   0.000000254209321
   0.406250000000000   0.000335861589750   0.000066251848244   0.000001211036203   0.000000254209492
   0.421875000000000   0.000335861589751   0.000066251848239   0.000001211036594   0.000000254209243
   0.437500000000000   0.000335861589753   0.000066251848246   0.000001211036446   0.000000254209457
   0.453125000000000   0.000335861589756   0.000066251848248   0.000001211036477   0.000000254209318
   0.468750000000000   0.000335861589760   0.000066251848251   0.000001211036491   0.000000254209393
   0.484375000000000   0.000335861589763   0.000066251848251   0.000001211036465   0.000000254209365
   0.500000000000000   0.000335861589766   0.000066251848251   0.000001211036490   0.000000254209370
   0.515625000000000   0.000335861589770   0.000066251848250   0.000001211036463   0.000000254209363
   0.531250000000000   0.000335861589772   0.000066251848247   0.000001211036488   0.000000254209391
   0.546875000000000   0.000335861589773   0.000066251848243   0.000001211036473   0.000000254209315
   0.562500000000000   0.000335861589773   0.000066251848241   0.000001211036440   0.000000254209454
   0.578125000000000   0.000335861589774   0.000066251848234   0.000001211036588   0.000000254209242
   0.593750000000000   0.000335861589772   0.000066251848239   0.000001211036198   0.000000254209491
   0.609375000000000   0.000335861589769   0.000066251848230   0.000001211036998   0.000000254209321
   0.625000000000000   0.000335861589765   0.000066251848222   0.000001211035652   0.000000254209150
   0.640625000000000   0.000335861589763   0.000066251848306   0.000001211037412   0.000000254210125
   0.656250000000000   0.000335861589762   0.000066251848022   0.000001211036170   0.000000254207674
   0.671875000000000   0.000335861589762   0.000066251848790   0.000001211034218   0.000000254212373
   0.687500000000000   0.000335861589766   0.000066251847167   0.000001211045175   0.000000254204962
   0.703125000000000   0.000335861589764   0.000066251849850   0.000001211014911   0.000000254214415
   0.718750000000000   0.000335861589785   0.000066251847476   0.000001211079016   0.000000254206213
   0.734375000000000   0.000335861589752   0.000066251842889   0.000001210967897   0.000000254205148
   0.750000000000000   0.000335861589774   0.000066251874959   0.000001211119742   0.000000254230560
   0.765625000000000   0.000335861590192   0.000066251765767   0.000001210994529   0.000000254157396
   0.781250000000000   0.000335861587199   0.000066252048384   0.000001210887045   0.000000254307488
   0.796875000000000   0.000335861601218   0.000066251458772   0.000001211689086   0.000000254056565
   0.812500000000000   0.000335861550505   0.000066252391667   0.000001209349978   0.000000254400770
   0.828125000000000   0.000335861689161   0.000066251687616   0.000001214462111   0.000000254050611
   0.843750000000000   0.000335861484463   0.000066249574339   0.000001205341026   0.000000254165800
   0.859375000000000   0.000335860888088   0.000066262224381   0.000001218330026   0.000000254769847
   0.875000000000000   0.000335867724035   0.000066220797931   0.000001206275603   0.000000252650534
   0.890625000000000   0.000335830815817   0.000066325407397   0.000001201630631   0.000000257352913
   0.906250000000000   0.000335979153919   0.000066112646761   0.000001259540663   0.000000249038806
   0.921875000000000   0.000335515846932   0.000066435045713   0.000001079723269   0.000000261114658
   0.937500000000000   0.000336497203799   0.000066240542353   0.000001486084788   0.000000247688006
   0.953125000000000   0.000336518724573   0.000065308370558   0.000000737367909   0.000000255051423
   0.968750000000000   0.000322552939052   0.000070150990899   0.000001866100018   0.000000265756105
   0.984375000000000   0.000414811616076   0.000058313777261   0.000000602750092   0.000000248919766
   1.000000000000000   0.000000000000026   0.000000000000037   0.000000000000014   0.000000000000003
    } \HtwoErrors
    \addplot[cyan, mark=triangle,mark size=1.5, line width=0.02cm] table[x=t, y=p2] \HtwoErrors;
    \addplot[magenta, mark=pentagon,mark size=1.2, line width=0.02cm] table[x=t, y=p3] \HtwoErrors;
    \addplot[blue, mark=square, mark size=1.2,line width=0.02cm] table[x=t, y=p4] \HtwoErrors;
    \addplot[darkgreen, mark=diamond,mark size=1.5, line width=0.02cm] table[x=t, y=p5] \HtwoErrors;
\end{groupplot}
\end{tikzpicture}
\end{minipage}
\begin{minipage}{0.475\textwidth}
\begin{tikzpicture}
\begin{groupplot}[group style={group size=1 by 1},height=7cm,width=8cm, every axis label={font=\normalsize}, ylabel style={font=\footnotesize}]
        
    \nextgroupplot[xtick={0,0.2,0.4,0.6,0.8,1},
                ytick={0.1,0.001,0.00001,0.0000001,0.000000001,0.00000000001,0.0000000000001,0.000000000000001},
                xlabel={$t$},
                ymode=log,
                title=$|\mathcal{E}_{\Psi_{\bh}}(t) - \mathcal{E}_{\Psi_{\bh}}(0)|$,
                legend pos=south west, 
                legend style={nodes={scale=1, transform shape}},]
    \pgfplotstableread{
        t       p2              p3              p4              p5              
        
                   0                   0                   0                   0                   0
   0.015625000000000   0.000414811616055   0.000058313777224   0.000000602750105   0.000000248919764
   0.031250000000000   0.000322552939035   0.000070150990863   0.000001866100030   0.000000265756104
   0.046875000000000   0.000336518724559   0.000065308370525   0.000000737367920   0.000000255051423
   0.062500000000000   0.000336497203784   0.000066240542326   0.000001486084798   0.000000247688007
   0.078125000000000   0.000335515846916   0.000066435045689   0.000001079723281   0.000000261114660
   0.093750000000000   0.000335979153897   0.000066112646742   0.000001259540676   0.000000249038809
   0.109375000000000   0.000335830815792   0.000066325407381   0.000001201630644   0.000000257352916
   0.125000000000000   0.000335867724006   0.000066220797917   0.000001206275617   0.000000252650538
   0.140625000000000   0.000335860888054   0.000066262224369   0.000001218330041   0.000000254769850
   0.156250000000000   0.000335861484425   0.000066249574329   0.000001205341043   0.000000254165805
   0.171875000000000   0.000335861689121   0.000066251687605   0.000001214462126   0.000000254050614
   0.187500000000000   0.000335861550464   0.000066252391655   0.000001209349995   0.000000254400772
   0.203125000000000   0.000335861601179   0.000066251458760   0.000001211689101   0.000000254056565
   0.218750000000000   0.000335861587161   0.000066252048371   0.000001210887059   0.000000254307487
   0.234375000000000   0.000335861590160   0.000066251765753   0.000001210994541   0.000000254157394
   0.250000000000000   0.000335861589747   0.000066251874946   0.000001211119752   0.000000254230557
   0.265625000000000   0.000335861589731   0.000066251842875   0.000001210967905   0.000000254205144
   0.281250000000000   0.000335861589769   0.000066251847464   0.000001211079021   0.000000254206208
   0.296875000000000   0.000335861589752   0.000066251849840   0.000001211014915   0.000000254214408
   0.312500000000000   0.000335861589756   0.000066251847157   0.000001211045177   0.000000254204956
   0.328125000000000   0.000335861589753   0.000066251848784   0.000001211034219   0.000000254212367
   0.343750000000000   0.000335861589752   0.000066251848020   0.000001211036172   0.000000254207670
   0.359375000000000   0.000335861589750   0.000066251848305   0.000001211037414   0.000000254210122
   0.375000000000000   0.000335861589749   0.000066251848225   0.000001211035655   0.000000254209148
   0.390625000000000   0.000335861589749   0.000066251848235   0.000001211037003   0.000000254209321
   0.406250000000000   0.000335861589750   0.000066251848244   0.000001211036203   0.000000254209492
   0.421875000000000   0.000335861589752   0.000066251848238   0.000001211036593   0.000000254209244
   0.437500000000000   0.000335861589754   0.000066251848246   0.000001211036445   0.000000254209457
   0.453125000000000   0.000335861589756   0.000066251848247   0.000001211036477   0.000000254209318
   0.468750000000000   0.000335861589759   0.000066251848250   0.000001211036490   0.000000254209393
   0.484375000000000   0.000335861589763   0.000066251848251   0.000001211036465   0.000000254209364
   0.500000000000000   0.000335861589766   0.000066251848251   0.000001211036489   0.000000254209370
   0.515625000000000   0.000335861589769   0.000066251848249   0.000001211036462   0.000000254209363
   0.531250000000000   0.000335861589772   0.000066251848247   0.000001211036486   0.000000254209391
   0.546875000000000   0.000335861589773   0.000066251848242   0.000001211036471   0.000000254209316
   0.562500000000000   0.000335861589774   0.000066251848240   0.000001211036440   0.000000254209454
   0.578125000000000   0.000335861589773   0.000066251848233   0.000001211036587   0.000000254209241
   0.593750000000000   0.000335861589771   0.000066251848239   0.000001211036197   0.000000254209491
   0.609375000000000   0.000335861589768   0.000066251848230   0.000001211036998   0.000000254209321
   0.625000000000000   0.000335861589766   0.000066251848222   0.000001211035651   0.000000254209150
   0.640625000000000   0.000335861589763   0.000066251848305   0.000001211037411   0.000000254210125
   0.656250000000000   0.000335861589762   0.000066251848022   0.000001211036170   0.000000254207674
   0.671875000000000   0.000335861589762   0.000066251848789   0.000001211034217   0.000000254212373
   0.687500000000000   0.000335861589765   0.000066251847166   0.000001211045175   0.000000254204962
   0.703125000000000   0.000335861589763   0.000066251849850   0.000001211014911   0.000000254214414
   0.718750000000000   0.000335861589784   0.000066251847475   0.000001211079015   0.000000254206213
   0.734375000000000   0.000335861589751   0.000066251842889   0.000001210967896   0.000000254205149
   0.750000000000000   0.000335861589774   0.000066251874959   0.000001211119742   0.000000254230560
   0.765625000000000   0.000335861590192   0.000066251765767   0.000001210994529   0.000000254157396
   0.781250000000000   0.000335861587199   0.000066252048384   0.000001210887045   0.000000254307488
   0.796875000000000   0.000335861601219   0.000066251458771   0.000001211689086   0.000000254056564
   0.812500000000000   0.000335861550505   0.000066252391666   0.000001209349979   0.000000254400770
   0.828125000000000   0.000335861689161   0.000066251687615   0.000001214462110   0.000000254050610
   0.843750000000000   0.000335861484463   0.000066249574340   0.000001205341025   0.000000254165801
   0.859375000000000   0.000335860888088   0.000066262224380   0.000001218330025   0.000000254769846
   0.875000000000000   0.000335867724035   0.000066220797930   0.000001206275602   0.000000252650534
   0.890625000000000   0.000335830815817   0.000066325407397   0.000001201630630   0.000000257352912
   0.906250000000000   0.000335979153918   0.000066112646762   0.000001259540662   0.000000249038806
   0.921875000000000   0.000335515846932   0.000066435045713   0.000001079723268   0.000000261114657
   0.937500000000000   0.000336497203798   0.000066240542353   0.000001486084787   0.000000247688005
   0.953125000000000   0.000336518724573   0.000065308370557   0.000000737367908   0.000000255051422
   0.968750000000000   0.000322552939052   0.000070150990898   0.000001866100017   0.000000265756104
   0.984375000000000   0.000414811616076   0.000058313777261   0.000000602750091   0.000000248919765
   1.000000000000000   0.000000000000026   0.000000000000036   0.000000000000015   0.000000000000002     
    } \HoneErrors
    \addplot[cyan, mark=triangle, mark size=1.5, line width=0.02cm] table[x=t, y=p2] \HoneErrors;
    \addplot[magenta, mark=pentagon, mark size=1.2, line width=0.02cm] table[x=t, y=p3] \HoneErrors;
    \addplot[blue, mark=square, mark size=1.2, line width=0.02cm] table[x=t, y=p4] \HoneErrors;
    \addplot[darkgreen, mark=diamond,  mark size=1.5,line width=0.02cm] table[x=t, y=p5] \HoneErrors;
\end{groupplot}
\end{tikzpicture}
\end{minipage}
\caption{Relative errors of the energy (left plot) and~mass (right plot) functionals with maximal regularity splines in space and time for~$p=2$ ({\footnotesize{\textcolor{cyan}{$\triangle$}}} marker),~$p=3$ ({\footnotesize{\textcolor{magenta}{$\pentagon$}}} marker),~$p=4$ ({\footnotesize{\textcolor{blue}{\(\square\)}}} marker) and~$p=5$ ({\footnotesize{\textcolor{darkgreen}{\(\lozenge\)}}} marker).}
\label{fig:6}
\end{figure}

\paragraph{Conservation.}
Due to the rapid decay of the exact solution close to the 
boundary $\partial \Omega \times (0, T)= \{-3,3\} \times (0, T)$ (see~\cite[Fig.~8 (panel a)]{Gomez_Moiola:2024}), the energy~\eqref{eq:2.2} and the mass~\eqref{eq:2.3} are expected to be conserved on the continuous level. In Figure~\ref{fig:6}, we report the relative difference between the energy and mass with respect to their initial discrete values. We observe that both quantities remain approximately constant over time, and their relative difference becomes exactly zero at $T = 1$, as stated in Proposition~\ref{prop:2.3}.

\section{Conclusions}
In this paper, we analyze a space--time isogeometric method for the linear Schrödinger equation, combining splines of maximal regularity in time with conforming spatial discretizations. The method is proven to be unconditionally stable and preserves both mass and energy invariants at the final time. The use of splines in time with maximal regularity ensures high accuracy with fewer degrees of freedom than for standard polynomial spaces. Key advancements of this work include:
\begin{itemize}
\item a theoretical proof of unconditional stability via the analysis of Toeplitz-like system matrices, without relying on numerical verifications;
\item a connection to a first-order-in-time variational formulation for the wave equation, enabling a new theoretical stability proof for the latter;
\item numerical experiments demonstrating optimal convergence rates in space and time, as well as conservation of physical invariants.
\end{itemize}

A complete variational analysis to study the stability properties and convergence estimates of the proposed method remains an open problem. This is primarily due to the current inability to handle the larger support of maximal regularity splines, as well as the fact that~$\partial_t S^p_{h_{t}}(0,T) = S^{p-1}_{h_{t}}(0,T) \not\subset S^p_{h_{t}}(0,T)$.

\section*{Acknowledgments}
The authors would like to acknowledge the kind hospitality of the Centre International de Rencontres Math\'ematiques (CIRM), where part of this research was conducted under the project \textit{``Unconditionally stable conforming space--time methods for the Schr\"odinger equation"}. This research was partially supported by the Austrian Science Fund (FWF) project  \href{https://doi.org/10.55776/P33477}{10.55776/P33477} (MF).
Both authors are members of the Gruppo Nazionale Calcolo Scientifico-Istituto Nazionale di Alta Matematica (GNCS-INdAM).

\noindent

\bibliography{references}{}
\bibliographystyle{plain}

\end{document}